\documentclass[10pt]{article}
\newcommand{\ddate}{}

\date{\ddate}
\usepackage[all,cmtip]{xy}
\usepackage{tikz-cd}

\usepackage{amssymb,amsthm}
\usepackage{graphicx}
\usepackage{pstricks}
\usepackage{psfrag}

\usepackage{hyperref}

\usepackage{epic}
\usepackage{eepic}

\newtheorem{dummy}{anything}[section]
\newtheorem{Theorem}[dummy]{Theorem}
\newtheorem{Lemma}[dummy]{Lemma}
\newtheorem{Proposition}[dummy]{Proposition}

\newtheorem{Corollary}[dummy]{Corollary}
\newtheorem{Definition}[dummy]{Definition}
\newtheorem{Definitions}[dummy]{Definitions}

\newtheorem{Examples}[dummy]{Examples}
\newtheorem{Remark}[dummy]{Remark}
\newtheorem{Remarks}[dummy]{Remarks}

\newtheorem{Problem}[dummy]{Problem}

\newtheorem{ccote}[dummy]{}

\newtheorem*{defA}{Definitions A}
\newtheorem*{defB}{Definitions B}

\newtheorem{Convention}[dummy]{Convention}

\newcommand{\bbr}{{\mathbb R}}
\newcommand{\bbp}{{\mathbb P}}

\newcommand{\bbz}{{\mathbb Z}}

\newcommand{\bbe}{{\mathbb E}}

\newcommand{\bbt}{{\mathbb T}}
\newcommand{\bba}{{\mathbb A}}

\newcommand{\bbu}{{\mathbb U}}

\newcommand{\calb}{{\mathcal B}}
\newcommand{\calc}{{\mathcal C}}
\newcommand{\cald}{{\mathcal D}}

\newcommand{\calf}{{\mathcal F}}

\newcommand{\calh}{{\mathcal H}}

\newcommand{\calm}{{\mathcal M}}
\newcommand{\caln}{{\mathcal N}}

\newcommand{\calp}{{\mathcal P}}
\newcommand{\calq}{{\mathcal Q}}

\newcommand{\cals}{{\mathcal S}}

\newcommand{\mancqfd}{\hfill \ensuremath{\Box}}

\newcommand{\pcirc}{\kern .7pt {\scriptstyle \circ} \kern 1pt}
\newcommand{\tpcirc}{\kern .7pt {\scriptstyle \tilde\circ} \kern 1pt}
\newcommand{\mun}{{-1}}

\newcommand{\psetminus}{{\scriptstyle\setminus}}
\renewcommand{\:}{\colon}

\newcommand{\beq}[1]{\begin{equation}\label{#1}}
\newcommand{\eeq}{\end{equation}}

\newcommand{\sk}[1]{\vskip #1 mm}
\newcommand{\nsk}[1]{\vskip #1 mm\noindent}
\newcommand{\eqref}[1]{(\ref{#1})}
\newcommand{\hfl}[2]{\smash{\mathop{\hbox to 1 truecm{\kern %
3pt\rightarrowfill\kern 3pt}}%
\limits^{\scriptstyle#1}_{\scriptstyle#2}}}
\newcommand{\cqfd}{\unskip\kern 6pt\penalty 500%
\raise -2pt\hbox{\vrule\vbox to10pt{\hrule width %
4pt\vfill\hrule}\vrule}\smallskip}
\newcommand{\proref}[1]{Proposition~\ref{#1}}
\newcommand{\remref}[1]{Remark~\ref{#1}}
\newcommand{\lemref}[1]{Lemma~\ref{#1}}
\newcommand{\corref}[1]{Corollary~\ref{#1}}
\newcommand{\thref}[1]{Theorem~\ref{#1}}

\newcommand{\secref}[1]{Section~\ref{#1}}

\newcommand{\dfn}[1]{{\it #1}}

\newcommand{\dia}[1]{\begin{array}{c}{\xymatrix@C-3pt@M+2pt@R-4pt{#1 }}\end{array}}

\newcommand{\comeq}[1]{\hbox{{\footnotesize #1}}}
\newcommand{\fl}[1]{\buildrel{#1}\over{\longrightarrow}}

\newcommand{\wh}{{\rm Wh}}

\newcommand{\rpi}{\bbr P^\infty}

\newcommand{\id}{{\rm id}}
\newcommand{\onto}{\to\kern-6.5pt\to}
\newcommand{\cs}{cohomology silent }
\newcommand{\has}[1]{cohomology $#1$-antisimple} 
\newcommand{\as}[1]{$#1$-antisimple} 
\newcommand{\bd}{{\rm Bd}}
\renewcommand{\bd}{\partial}
\newcommand{\wa}{{\rm Wall}}
\newcommand{\pe}{Poincar\'e\ }
\newcommand{\dcup}{\amalg}
\newcommand{\ecs}[2]{{\rm CS}(#1,#2)}
\newcommand{\cef}[2]{($#1,#2$)-cellfree}
\newcommand{\ccs}[2]{cohomology ($#1,#2$)-silent}

\title{The cell-dispensability obstruction \\ for spaces and manifolds}
\author{Jean-Claude HAUSMANN}

\begin{document}

\maketitle
 
\begin{abstract}
We compare two properties for a CW-space $X$ of finite type: 
(1) being homotopy equivalent to a CW-complex without $j$-cells for $k\leq j\leq \ell$ ({\it \cef{k}{\ell}}) and 
(2) $H^j(X;R)=0$ for any $\bbz\pi_1(X)$-module~$R$ when $k\leq j\leq \ell$ ({\it\ccs{k}{\ell}}). 
Using the technique of Wall's finiteness obstruction, we show that a connected CW-space $X$ of finite type which is \ccs{k}{\ell} determines a \dfn{cell-dispensability obstruction} $w_k(X)\in\tilde K_0(\bbz\pi_1(X))$ which vanishes if and only if $X$ is \cef{k}{\ell} ($k\geq 3$). Any class in $\tilde K_0(\bbz\pi)$ may occur as the cell-dispensability obstruction $w_k(X)$ for a CW-space $X$ with $\pi_1(X)$ identified with $\pi$. Using projective surgery, a similar theory is obtained for manifolds, replacing ``cells'' by ``handles'' (antisimple manifolds). 
\end{abstract}

{\footnotesize \tableofcontents }

\section{Introduction}

In the first part of this paper, we are interested in the following possible properties 
for a connected CW-space $X$.

\begin{defA}\rm
\begin{itemize}
\item[(1)] $X$ is \dfn{\cef{k}{\ell}} if $X$ is homotopy equivalent to a CW-complex 
having finite skeleta and without $j$-cells for $k\leq j\leq\ell$. 
\item[(2)]  $X$ is \dfn{\cs in degree $j$} if $H^j(X;R)=0$ for any $\bbz\pi_1(X)$-module~$R$.
\item[(3)]  $X$ is \dfn{\ccs{k}{\ell}} if $X$ is \cs in degree $j$ for~$k\leq j\leq\ell$.
\end{itemize}
\end{defA}

Note that a \cef{k}{\ell} CW-space $X$ is of finite type (i.e. homotopy equivalent to a complex with finite skeleta). But the definition is stronger than being of finite type {\it and} being homotopy equivalent to a complex without $j$-cells for $k\leq j\leq\ell$. Indeed, the two conditions must be realized simultanously
(see \remref{remcellfree}).

Obviously, a CW-space which is \cef{k}{\ell} is \ccs{k}{\ell}. The converse is known to be true for simply connected CW-spaces of finite type, using a minimal cell decomposition \cite[Proposition~4C.1]{Hatcher} together with the universal coefficient theorem for cohomology. In \secref{GSCSC}, we show that a connected CW-space $X$ of finite type which is \ccs{k}{\ell} determines a \dfn{cell-dispensability obstruction} $w_k(X)\in\tilde K_0(\bbz\pi_1(X))$ which vanishes if and only if $X$ is \cef{k}{\ell} ($k\geq 3$). Any class in $\tilde K_0(\bbz\pi)$ may occur as the cell-dispensability obstruction $w_k(X)$ for a CW-space $X$ with $\pi_1(X)$ identified with $\pi$ (by somehow fixing the ($k-1$)-skeleton).

The definition of $w_k(X)$ is akin to that of the Wall finiteness obstruction, but the two obstruction are different,
since $w_k(X)$ might not vanish for a finite complex $X$.
Among the similarities is a product formula. As a consequence, if $X$ is \ccs{k}{\ell} ($k\geq 3$), then 
$X\times S^1$ is \cef{k+1}{\ell}.

In \secref{GScasm}, we apply the above material to smooth manifolds, using the following
\begin{defB}\rm
Let $M$ be a smooth compact manifold of dimension~$r\geq 2k$.
\begin{itemize}
\item[(1)] $M$ is \dfn{\as{k}} \ if $M$ admits a handle decomposition without handles of index $j$ 
for $k\leq j \leq r-k$.
\item[(2)] $M$ is \dfn{\has{k}} if it is \ccs{k}{r-k}. This definition is also used for a \pe complex
of formal dimension $r$. 
\end{itemize}
\end{defB}

If $M$ is \has{k}, its cell-dispensability obstruction $w_k(M)\in\tilde K_0(\bbz\pi_1(M))$ is defined.
When $M$ is closed of dimension $r\geq 6$, we prove that $w_k(M)=0$ if and only if $M$ is \as{k}. The 
obstruction $w_k(M)$ is $r$-self-dual, i.e $w_k(M)=(-1)^{r+1}w_k(M)^*$ (the involution $*$ on $\tilde K_0(\bbz\pi_(M)$
involves the orientation character $\omega\:\pi_1(M)\to\{\pm 1\}$ of $M$).
It thus defines a class 
$[w_k(M)]$ in the Tate cohomology group  $H^{r+1}(\bbz_2;\tilde K_0(\bbz\pi_(M),\omega))$. We prove that 
$[w_k(M)]=0$ if and only if $M$ is cobordant, in an appropriate sense, to a \as{k} manifold.

A whole section (\S~\ref{Srea}) is devoted to realize an element $\calp\in\tilde K_0(\bbz\pi)$ as the
cell-dispensability obstruction for a \pe space or a closed manifold with fundamental group identified with $\pi$
(via fixing the normal ($k-1$)-type). Using the technique of ends of spaces \cite{PeRa,RanEnds}, we construct for any 
$\calp\in\tilde K_0(\bbz\pi)$ a finitely dominated \pe space $\bbp$ of formal dimension $r\geq 6$, 
which is \has{k} with $w_k(\bbp)=\calp$. Its Wall finiteness obstruction $\wa(\bbp)$ is equal to $\calp + (-1)^{r}\calp^*$. In particular, when $\calp$ is $r$-self-dual,
$\bbp$ is homotopy equivalent to a finite \pe complex. Using then the projective surgery theory of \cite{PeRa},
we prove that $\calp$ may be replaced by a closed $r$-dimensional manifold when $[w_k(M)]$
belongs to the kernel of the homomorphism  $\delta_R\:H^{r+1}(\bbz_2;\tilde K_0(\bbz\pi,\omega))\to L_{r}^h(\pi,\omega)$
occurring in the Ranicki exact sequence (see~\eqref{ranicki}). 
Examples are discussed in \S~\ref{Sexpl}. Finally, \secref{SrefASManif} contains some results about antisimple manifolds, shedding some new light on previous works of the author \cite{HauMW,HauTWI,Hau2}.

The terminology {\it antisimple} was introduced in \cite{HauNAS} for manifolds related to high dimensional knots. 
Antisimple manifolds have been studied in \cite{HauMW}, with applications in \cite{HauTWI,Hau2}. 
In \cite{DFW,WeinbHigh},  {\it cohomology antisimple manifolds} are called  
{\it antisimple manifold}.

\sk{3}\noindent{\bf Acknowledgments: }\rm The author is grateful to Ian Hambleton and Shmuel Weinberger for comments and suggestions.

\section{Cohomology silent complexes}\label{GSCSC}

\subsection{Preliminaries}\label{Snot}

\begin{ccote} Cell complexes. \rm\ 
A CW-complex is also called a \dfn{cell complex}
or just a \dfn{complex}. We denote by $X^k$ the $k$-skeleton of a complex $X$. 
The universal covering of $X$ is denoted by $\tilde X$. Maps between topological spaces are always 
supposed to be continuous.

When a cell complex $X$ is connected, we usually assume that it has only one $0$-cell: 
$X^0=\{x_0\}$ and we denote by $\pi(X)$ the fundamental group $\pi_1(X,x_0)$, often abbreviated by $\pi$
when the context is clear. The group $\pi$ acts on the left on $\tilde X$ by deck transformations.

A \dfn{CW-space} is a topological space $X$ which is homotopy equivalent to a CW-complex.
Such a CW-space is 
\begin{itemize}
 \item of \dfn{finite $k$-type} if $X$ is homotopy equivalent to a CW-complex with finite $k$-skeleton.
 \item of \dfn{finite type} if $X$ is of finite $k$-type for all $k\geq 0$. 
 \item \dfn{finitely dominated} if there exists a finite CW-complex $A$ and maps $f\:A\to X$ and $s\:X\to A$
 such that $f\pcirc s$ is homotopic to ${\rm id}_A$ ($f$ is called a \dfn{(homotopy) domination}).  
\end{itemize} 
\end{ccote}

\begin{ccote}\label{connectivity}  $r$-connectivity and weak $r$-connectivity. \rm\ 
A map $f\:Y\to X$ between connected CW-spaces is called \dfn{$r$-connected} if $\pi_j(f)=0$ for $j\leq r$. 
When $r\geq 2$, this is equivalent to $f$ inducing an isomorphism on fundamental groups and $H_j(\tilde f)=0$
for $j\leq r$, where $\tilde f\:\tilde Y\to\tilde X$ is the induced map on the universal covers. 
A pair $(X,Y)$ of connected CW-spaces is called \dfn{$r$-connected} if the inclusion map $Y\hookrightarrow X$
is $r$-connected.

A map $f\:Y\to X$ is called \dfn{weakly $r$-connected} if $f_*\: \pi_j(Y)\to\pi_j(X)$ 
is an isomorphism for $j\leq r-1$. When this is the case, one can attach
trivial $r$-cells to $Y$, getting a cell complex $Y'$, so that $f$ extends to $f'\:Y'\to X$ which is $r$-connected.
Again, a pair $(X,Y)$ of connected CW-spaces is called \dfn{weakly $r$-connected} if the inclusion map $Y\hookrightarrow X$
is weakly $r$-connected. Note that ``weakly $r$-connected'' implies ``($r-1$)-connected''.

Well known by specialists, the following lemma will be useful.

\begin{Lemma}\label{r-connSkel}
Let $f\:Y\to X$ be a map between connected CW-complexes. Suppose that $f$ is $r$-connected.
Then, there exists a CW-complex $Y'$ obtained by attaching to $Y$ cells of dimension $>r$, such that 
$f$ extends to a homotopy equi\-valence $f'\:Y'\to X$. If $X$ and $Y$ have finite skeleta, then $Y'$ may be chosen to have finite skeleta.
\end{Lemma}

\begin{proof} 
It suffices to prove that there exists a CW-complex $\bar Y$ obtained by attaching to $Y$ cells of dimension $r+1$, 
finite in number if $X$ and $Y$ have finite skeleta,
such that $f$ extends to an $(r+1)$-connected $\bar f\:\bar Y\to X$. The required couple $(Y',f')$ will the 
be obtained by iteration of this process. 

One can replace $f$ by an injection of a subcomplex.
It is classical that the couple $(\bar Y,\bar f)$ may be obtained by attaching ($r+1)$-cells using a set of generators of $\pi_{r+1}(X,Y)$ (see \cite[p.~59, comments preceeding Lemma~1.2 and its proof]{WallFin}).
By the homotopy sequence of the triple $(X,Y,Y^r)$, the $\bbz\pi_1(X)$-module $\pi_{r+1}(X,Y)$ is a quotient of
$\pi_{r+1}(X,Y^r)$.The latter is a finitely generated $\bbz\pi_1(X)$-module when $X$ and $Y$ are complexes with finite skeleta (see \cite[Theorem~A]{WallFin}).
\end{proof}
\end{ccote}

\begin{ccote}\label{reference}  Reference maps. \rm\ 
Let $Y$ be a connected CW-complex of dimension $r-1$. An \dfn{$Y$-reference map} for a CW-space $X$ is a map 
$g\:X\to Y$ which is $r$-connected. The space $X$ together with the map $g$ (or the pair $(X,g)$) is called an 
\dfn{$Y$-referred space}. 

\begin{Lemma}\label{Lrefretr}
Let $g\:X\to Y$ be an $Y$-reference map. Then, there exists a CW-complex $X'$ containing $Y$ as a subcomplex 
and a homotopy equivalence $h\:X'\to X$ such that the $Y$-reference map $g\pcirc h$ is a retraction of $X'$ onto $Y$.
The pair $(X',Y)$ is weakly $r$-connected.
\end{Lemma}

In short, a reference map may be realized up to homotopy equivalence by a {\it reference retraction}. Note that the pair $(X',Y)$ is not $r$-connected  in general (e.g. $X'=X=Y\vee S^r$). 

\begin{proof}
The reference map $g$ may be, up to homotopy equivalence, replaced by a Serre fibration $\hat g\:\hat X\to Y$, 
whose fiber is then ($r-1$)-connected. As $Y$ is of dimension $r-1$, there is no obstruction to get a section 
of $\hat g$, which provides a homotopy section $\gamma\:Y\to X$ for $g$. 
This map $\gamma$ may be, up to homotopy equivalence, 
replaced by a CW-inclusion $J\:Y \hookrightarrow X'$ and the map $g$ provides a 
reference map $g_1'\:X'\to Y$, such that
$g_1'\pcirc J$ is homotopic to ${\rm id}_Y$. By the homotopy extension property, there is a homotopy from $g_1'$ to 
$g'\:X'\to Y$ such that $g'\pcirc\gamma={\rm id}_Y$, namely $g$ is a retraction. Since $g$ is $r$-connected,
the pair $(X',Y)$ is weakly $r$-connected.
\end{proof}
\end{ccote}

\begin{ccote}\label{cohmod} Cohomology with local coefficients. \rm 
Let $X$ be a connected cell complex with fundamental group $\pi$. We consider the cellular chain complex 
$(C_*(\tilde X),\partial)$, with $C_k(\tilde X) = H_k(\tilde X^k,\tilde X^{k-1})$. The latter is a
is a free (left) $\bbz\pi$-module 
with basis in bijection with the $k$-cells of $X$ (by choosing, for each $k$-cell $e$, 
an orientation and a lifting $\tilde e$ in $\tilde X$ of $e$). The \dfn{cycles} $Z_k(\tilde X)$ and 
the \dfn{boundaries} $B_{k-1}(\tilde X)$ are, as usual, defined as the $\bbz\pi$-modules
\begin{eqnarray*}
Z_k(\tilde X) &=& \ker\big(\partial \: C_k(\tilde X)\to C_{k-1}(\tilde X)\big) \\
B_{k-1}(\tilde X) &=& {\rm Image}\big(\partial \: C_k(\tilde X)\to C_{k-1}(\tilde X)\big) \, .
\end{eqnarray*}

Modules of $\bbz\pi$ are {\it left} modules, unless mention of the contrary.
If $R$ is a such a $\bbz\pi$-module,
the \dfn{cohomology} $H^*(X;R)$ is the homology of the cochain complex $C^*(\tilde X;R)=\hom(C_*(\tilde X),R)$, where the coboundary
$\delta$ is defined by 
$$
\delta(\alpha) = (-1)^k \alpha\pcirc\partial \quad  (\alpha\in C^k(\tilde X))  \, .
$$

Homology with coefficient in a $\bbz\pi$-module $R$ will also be used, but only in the framework of manifolds or \pe spaces (see \ref{holoc}). 
\end{ccote}

\subsection{The cell-dispensability obstruction}\label{SCSC}

Recall from the introduction (Definition~A) that a 
connected CW-space $X$ is \dfn{\cs in degree $k$} if $H^k(X;R)=0$ for any $\bbz\pi_1(X)$-module $R$.

\begin{Lemma}\label{LW1}
Let $X$ be a connected CW-complex and let $k\geq 1$ be an integer.
The following conditions are equivalent.
\begin{itemize}
\item[(a)] $X$ is \cs in degree $k$.
\item[(b)] The following two conditions hold true:
   \begin{itemize}
   \item[(b1)]  the inclusion $B_{k-1}(\tilde X)\hookrightarrow C_{k-1}(\tilde X)$ admits a
                retraction of $\bbz\pi(X)$-modules (in consequence $B_{k-1}(\tilde X)$ is projective), and
   \item[(b2)] $H_k(\tilde X)=0$.
   \end{itemize}
\end{itemize}
\end{Lemma}

\begin{proof}
We set $C_j=C_j(\tilde X)$, $Z_j=Z_j(\tilde X)$, $B_j=B_j(\tilde X)$ and $\pi=\pi_1(X)$. The proof proceeds in three steps.

\sk{1}\noindent{\it Step 1: (a) $\Rightarrow$ (b1).} 
The boundary homomorphism $\partial\:C_k\to B_{k-1}$ may be seen as an element of 
$Z^k(X;B_{k-1}(\tilde X))$. If (a) holds true, there is $\alpha\in\hom(C_{k-1},B_{k-1})$ such that
$\partial=\delta(\alpha)$, i.e.
$\partial(u)=\alpha\pcirc\partial(u)$ for all $u\in C_k$. Thus, $\alpha$ is a retraction of $C_{k-1}$ onto $B_{k-1}$.

\sk{1}\noindent{\it Step 2: (a) $\Rightarrow$ (b2).} By Step 1 already proven, $B_{k-1}$ is projective and thus 
the exact sequence $0\to Z_k\to C_k\to B_{k-1}\to 0$ splits, implying that $Z_k$ is a direct summand of $C_k$.
Therefore, the surjective homomorphism $Z_k\to H_k(\tilde X)$ extends to a homomorphism $\alpha\:C_k\to H_k(\tilde X)$,
giving an element of $C^k(X;H_k(\tilde X))$ 
satisfying $\alpha\pcirc \partial=0$, i.e. $\alpha\in Z^k(X;H_k(\tilde X))$.
By (a), $\alpha = \bar\alpha\pcirc\partial$ for some homomorphism $\bar\alpha\:C_{k-1}\to H_k(\tilde X)$. 
The homomorphism $\bar\alpha$ is surjective since so is $\alpha$. But $\partial(Z_k)=0$, which implies that $H_k(\tilde X)=0$.

\sk{1}\noindent{\it Step 3: (b) implies (a).} Let $b\in H^k(X;R)$, represented by a homomorphism
$\beta\: C_k(\tilde X)\to R$ which is a cocycle, i.e. $\beta_{|B_k}=0$. But $B_k=Z_k$ since $H_k(\tilde X)=0$ by (b2).
Since $k\geq 1$, one has $\beta = \bar\beta\pcirc\partial$ for some morphism of $\bbz\pi$-module
$\bar\beta\:B_{k-1}\to R$. By (b1), $\bar\beta$ extends to
a an element $\hat\beta\in C^{k-1}(X;R)$ and thus $\beta = \delta(\hat\beta)$, which proves that $b=0$.
Therefore, $H^k(X,R)=0$ for any $\bbz\pi_1(X)$-module $R$, which proves (a).
\end{proof}

We now prepare the definition of the obstruction $w_k(X)$ associated to a \cs cell complex in degree $k$.

\begin{Lemma}\label{LW1b}
Let $X$ be a connected CW-complex which is \cs in some degree $k\geq 3$. Let $\varphi\:K\to X$ be a $(k-1)$-conneted map,
where $K$ is a $(k-1)$-dimensional CW-complex. Then, the $\bbz\pi$-module $\pi_k(\varphi)$ is projective.
\end{Lemma}

\begin{proof}
Up to homotopy equivalence, we may suppose that $\varphi$ is the inclusion of the $(k-1)$-skeleton $K=X^{k-1}\to X$ 
(see \lemref{r-connSkel}). 
Since $k\geq 3$, the map $\varphi$ induces an isomorphism on the fundamental groups and is covered by a
map $\tilde\varphi\:\tilde K \to \tilde X$ on the universal coverings. By the relative Hurewicz isomorphism theorem, 
one has
$$
\pi_k(\varphi) \approx H_k(\tilde\varphi) \approx H_k(\tilde X,\tilde X^{k-1}) \, .
$$
We now claim that there is an isomorphism of $\bbz\pi$-modules between $B_{k-1}(\tilde X)$ and 
$H_k(\tilde X,\tilde X^{k-1})$. \lemref{LW1b} will then follow from \lemref{LW1}.

To establish the claim, recall that the boundary homomorphism $\partial\: C_k(\tilde X)\to C_{k-1}(\tilde X)$ 
is the composition
$$
\partial : H_k(\tilde X^{k},\tilde X^{k-1}) \to H_{k-1}(\tilde X^{k-1}) \to H_{k-1}(\tilde X^{k-1},\tilde X^{k-2})  \, .
$$
The right arrow is injective since $H_{k-1}(\tilde X^{k-2})=0$. Hence, $B_{k-1}(\tilde X)$ may be seen as a submodule of 
$H_{k-1}(\tilde X^{k-1})$:
\small
$$
B_{k-1}(\tilde X) \approx {\rm Im\,}\big(H_k(\tilde X^{k},\tilde X^{k-1})\to H_{k-1}(\tilde X^{k-1})\big) = 
\ker\big(H_{k-1}(\tilde X^{k-1})\to H_{k-1}(\tilde X^{k})\big) \, .
$$
\rm
As $H_{k-1}(\tilde X^{k})\approx H_{k-1}(\tilde X)$, one has 
$$
B_{k-1}(\tilde X) \approx
\ker\big(H_{k-1}(\tilde X^{k-1})\to H_{k-1}(\tilde X)\big) =  
{\rm Im\,}\big(H_k(\tilde X,\tilde X^{k-1}) \to H_{k-1}(\tilde X^{k-1})\big) \, .
$$
But
$$
{\rm Im\,}\big(H_k(\tilde X,\tilde X^{k-1}) \to H_{k-1}(\tilde X^{k-1})\big) \approx H_k(\tilde X,\tilde X^{k-1})
$$
since $H_k(\tilde X)=0$ by \lemref{LW1}.
\end{proof}

Let $X$ be a connected CW-complex which is \cs in some degree $k\geq 3$. 
Let $\varphi\:K^{k-1}\to X$ be a $(k-1)$-connected map. By \lemref{LW1b}, $\pi_k(\varphi)$ is a projective 
$\bbz\pi$-module. Suppose that $K$ is a finite complex and that $X$ is of finite type. By the proof of \lemref{LW1b},
the $\bbz\pi$-module $\pi_k(\varphi) \approx H_k(\tilde\varphi)$ is finitely generated and therefore defines 
a class 
\beq{defwk}
w_k(X) = (-1)^k\,[\pi_k(\varphi)] \in \tilde K_0(\bbz\pi) \, ,
\eeq
where $K_0(\bbz\pi)$ denotes the group of finitely generated projective $\bbz\pi$-modules modulo the stably free ones.
That the map $\varphi$ does not appear in the notation $w_k(X)$ makes sense in view of the following lemma.

\begin{Lemma}\label{WL2b}
Let $X$ be a connected CW-complex which is of finite type and \cs in degree $k\geq 3$. 
Let $\varphi_i\:K_i^{k-1}\to X$ ($i=1,2$) be two ($k-1$)-connected maps, where $K_i$ are finite complexes. Then 
$[\pi_k(\varphi_1)] = [\pi_k(\varphi_2)]$ in $\tilde K_0(\bbz\pi)$.
\end{Lemma}

\begin{Remarks}\label{WR1}\rm
(1) When $X$ has the homotopy type of a $k$-dimensional complex, then 
the Wall finiteness obstruction $\wa(X)$ of $X$ is defined \cite{WallFin} and is equal to $w_k(X)$.
See also \secref{SWallFO}.
\sk{1}\noindent (2)
Taking for $\varphi\:K\to X$ the inclusion of $K=X^{k-1}$ into $X$, one has $w_k(X)=(-1)^k[B_{k-1}(X)]$
(see the proof of \lemref{LW1b}).
\mancqfd\end{Remarks}

\begin{proof}[Proof of \lemref{WL2b}] We essentially follow the proof of \cite[Lemma~3.2]{WallFin}, 
as detailed in \cite[Lemma~2.9, p.~155]{Vara}.
Up to homotopy equivalence, the map $\varphi_i$ may be replaced by a Serre fibration $\theta_i$, with fiber $F_i$.
As $\varphi_i$ is $(k-1)$-connected, one has $\pi_j(F_i)=0$ for $j<k-1$, and thus
the first obstruction to construct a section of $\theta_i$ belongs to 
$H^k(X;\pi_{k-1}(F_i))$. The latter vanishes since $X$ is \cs in degree $k$. Using obstruction theory 
\cite[Chapter VI]{WhiteheadBoo}, we can thus construct map $s^i\:X^k\to K_i$ 
such that $\varphi_i\pcirc s^i$ is homotopic to the inclusion 
$X^k\hookrightarrow X$. This induces a morphism of $\bbz\pi$-module 
$s^i_*\:H_{k-1}(\tilde X)\approx H_{k-1}(\tilde X^k)\to  H_{k-1}(\tilde K_i)$ which splits the exact sequence
\beq{Waeq1}
\dia{
0 \ar[r] & H_k(\tilde\varphi_i) \ar[r] & H_{k-1}(\tilde K_i)  \ar[r]^{\varphi_i}   &  H_{k-1}(\tilde X)   
\ar@/^2mm/[l]^{s^i_*}  \ar[r] & 0
}
\eeq
(the morphism $H_k(\tilde\varphi_i) \to H_{k-1}(\tilde K_i)$ is indeed injective since
$H_k(\tilde X)=0$ by \lemref{LW1}). 
By cellularity, one may assume that $\varphi_i(K_i)\subset X^{k-1}$. We can thus define a map 
$\psi\:K_1\to K_2$ as $\psi=s^2\pcirc \varphi_1$.

Since $\varphi_i$ is $(k-1)$-connected and since $K_i$ is of dimension $k-1$, one has $H_j(\tilde\psi)=0$
when $j\neq k-1,k$ and the horizontal line in the following diagram is an exact sequence.
{\small
$$
\dia{ 
H_k(\tilde\psi) \ar@{>->}[r] & H_{k-1}(\tilde K_1) \ar@{>>}[dr]^{(\tilde\phi_1)_*} \ar[rr]^{\tilde\psi_*}
&& H_{k-1}(\tilde K_2) \ar@{>>}[r]& 
H_{k-1}(\tilde\psi) \\
&& H_{k-1}(\tilde X)  \ar@{>->}[ur]^{\tilde s^2_*}
}
$$
}
Therefore, using \eqref{Waeq1}, one gets isomorphisms of $\bbz\pi$-modules
\beq{Waeq2}
H_k(\tilde\varphi_1) \approx H_k(\tilde\psi) \quad {\rm and } \quad H_k(\tilde\varphi_2) 
\approx H_{k-1}(\tilde\psi) \, .
\eeq

With the abbreviations $C_j=C_j(\tilde\psi)$, $Z_j=Z_j(\tilde\psi)$ and
$B_j=b_j(\tilde\psi)$, the cellular chain complex $C_*(\tilde\psi)$ gives the following exact sequences.

\sk{2}
\begin{tabular}{lcll}
$Z_1\hookrightarrow C_1 \onto C_0$ & $\Rightarrow$ & $B_1=Z_1$ is stably free \\[1mm]
$Z_2\hookrightarrow C_2 \onto B_1$ & $\Rightarrow$ & $B_2=Z_2$ is stably free, {\it etc.}\\
\end{tabular}
\sk{1}\noindent
This shows that $B_j=Z_j$ is stably free for $j < k-1$ and that $Z_{k-1}$ is stably free. Then one has
\beq{Waeq3}
B_{k-1}\hookrightarrow Z_{k-1} \onto H_{k-1}(\tilde\psi) \quad , \quad
Z_k\hookrightarrow C_k \onto B_{k-1} \quad , \quad  Z_k \fl{\approx} H_{k}(\tilde\psi) \, .
\eeq
Since $X$ is \cs in degree $k$, the $\bbz\pi$-modules $H_{k-1}(\tilde\psi)$ and $H_{k-1}(\tilde\psi)$ are
projective by \lemref{LW1} and \eqref{Waeq2}. By \eqref{Waeq3}, $B_{k-1}$ and $Z_k$ are also projective and,
by \eqref{Waeq2} again,
$$
[H_k(\tilde\varphi_2)]=[H_{k-1}(\tilde\psi)]=-[B_{k-1}]=[Z_k]=
 [H_k(\tilde\psi)] = [H_k(\tilde\varphi_1)] \, .  \qedhere
$$
\end{proof}

A slight modification of the proof of \lemref{WL2b} shows that $w_k(X)$ is an invariant of the $k$-type of $X$.
More precisely, one has the following

\begin{Lemma}\label{WL2bb}
Let $X_i$ (i=1,2) be two connected CW-complexes of finite type and $k\geq 3$ be an integer.
Suppose that, for $i=1,2$, $X_i$ is \cs in degree $k$.
Let $h\:X_1\to X_2$ be a $k$-connected map. Then, $w_k(X_2)=h_*(w_k(X_1))$, where 
$h_*\:\tilde K_0(\bbz\pi_1(X_1))\fl{\approx}\tilde K_0(\bbz\pi_1(X_2))$ is the isomorphism induced by $h$.
\end{Lemma}

\begin{proof}
We assume that $h$ is cellular. We follow the proof of \lemref{WL2b} with the map $\varphi_i\:K_i\to X_i$
being the inclusion of $K_i=X_i^{k-1}$ into $X$. For $i=1,2$, we get a map $s_i\:X_i^{k}\to K_i$ 
such that $\varphi_i\pcirc s_i$ is homotopic to the inclusion 
$X_i^k\hookrightarrow X_i$. We then just need to replace the map $\psi$ of the proof of \lemref{WL2b} by 
$\psi=s_2\pcirc h\pcirc \varphi_1$. 
\end{proof}

\begin{Remark}\label{RCWspace}\rm
\lemref{WL2bb} enables us to define $w_k(X)\in\tilde K_0(\pi_1(X))$ for a {\it CW-space} $X$ 
of finite $k$-type which is \cs in degree $k$.
Just choose a homotopy equivalence $f\:X'\to X$ where $X'$ is a CW-complex with finite $k$-skeleton
and define $w_k(X)=f_*(w_k(X'))$. \lemref{WL2bb} guarantees that this is undependent of the choice of $(X',f)$.   
\mancqfd\end{Remark}

The element  $w_k(X)$ will be called the \dfn{cell-dispensability obstruction of} $X$,
thanks to the following proposition, in which the the concepts of Definitions~A of the introduction are used.

\begin{Proposition}\label{WP1} 
Let $3\leq k\leq \ell$ be integers. Let $X$ be a connected CW-space
of finite type. Suppose that $X$ is \ccs{k}{\ell}.
Then, the following conditions are equivalent.
\begin{itemize}
 \item[(a)]  $w_k(X)=0$.
 \item[(b)] $X$ is \cef{k}{\ell}
 \end{itemize}
\end{Proposition}

\begin{proof}
Using \remref{RCWspace}, we assume that $X$ is a CW-complex with finite skeleta.
If $X$ has no $k$-cells, then $B_{k-1}(\tilde X)=0$ and thus $w_k(X)=0$ by Part (2) of \remref{WR1}. 
Therefore, (b) implies (a).

Conversely, suppose that (a) holds. As seen in Step 1 of the proof of \lemref{LW1}, the module $B_{k-1}(\tilde X)$ is a retract of $C_{k-1}(\tilde X)$. We may thus write $C_{k-1}(\tilde X)\approx B_{k-1}(\tilde X)\oplus T$. 
If $w_k(X)=0$, then $B_{k-1}(\tilde X)$ is a finitely generated stably free $\bbz\pi$-module, and so is $T$.
By wedging if necessary $X$ with a finite collection of $D^{k-1}$'s (canceling pairs of $(k-2)$ and $(k-1)$-cells),
one may assume that $T$ is $\bbz\pi$-free, with basis $\calb = \{t_1,\dots,t_r\}$. 

As $k\geq 3$, the complexes $\tilde X^{k-1}$ and $\tilde X^{k-2}$ are connected and thus the Hurewicz
homomorphism
$$
\dia{
\pi_{k-1}(X^{k-1},X^{k-2}) \ar@{>>}[r]^(0.4){h} & H_{k-1}(\tilde X^{k-1},\tilde X^{k-2}) = C_{k-1}(\tilde X)
}
$$
is surjective (see \cite[Theorem~7.2, p.~178]{WhiteheadBoo}. 
One can thus lift $T$ as a free submodule of 
$\pi_{k-1}(X^{k-1},X^{k-2})$ and represent $t_j$ by a map of pair $\tau_j\:(D^{k-1},S^{k-2})\to(X^{k-1},X^{k-2})$.
Let $K$ be the finite CW-complex of dimension $k-1$ obtained by adding $(k-1)$-cells 
$e_1,\dots,e_r$ to $X^{k-2}$, the cell $e_j$ being attached by the map $\tau_j\:S^{k-2}\to X^{k-2}$.
Thus, $C_{k-1}(\tilde K)=T$.
The inclusion $X^{k-2}\to X$ then extends to a map $f\: K\to X$, using the map $\tau_j$ on the cell~$e_j$. 
One checks that $f$ induces an isomorphism on
the fundamental groups and on the homology groups $H_i(\tilde K) \fl{\approx} H_i(\tilde X)$ for $i\leq k-1$.
But, by  \lemref{LW1}, $H_j(\tilde X)=0$ for $k\leq j\leq\ell$.
Therefore, $f$ is $\ell$-connected and, by \lemref{r-connSkel}, there exists a homotopy equivalence
$f'\:K'\fl{} X$ where $K'$ has finite skeleta and no $j$-cells for $k\leq j\leq\ell$.
\end{proof}

\begin{Remark}\label{Rj>k}\rm
In \proref{WP1}, one may wonder that Condition (b) is just $w_k(X)=0$ and not 
$w_j(X)=0$ for $k\leq j\leq\ell$. But, actually, $w_k(X)=w_{j}(X)$ for $k\leq j\leq l$. Indeed, as $H_j(\tilde X)=0$
for $k\leq j\leq\ell$, one has $B_j=Z_j$ and thus an exact sequences $0\to B_j \to C_j \to B_{j-1} \to 0$ for 
$k\leq j\leq\ell$.
\mancqfd\end{Remark}

\begin{Remark} Multiple gaps. \  \rm
Let $3\leq k_1\leq \ell_1 < k_2\leq \ell_2 <\cdots< k_q\leq \ell_q$ be integers. One may consider a CW-space of finite type which is \ccs{k_j}{\ell_j} for all $j$ with $1\leq j\leq q$. Then the cell-dispensability obstructions $w_{k_j}(X)\in\tilde K_0(\bbz\pi_1(X))$
are defined for $1\leq j\leq q$ and \proref{WP1} is valid for all $j$ simultaneously. For instance, if $w_{k_j}(X)=0$
for all $j$, then $X$ is homotopy equivalent to a CW-complex having finite skeleta and without $r$-cells for $k_j\leq r\leq\ell_j$, for all $j$. 
\mancqfd\end{Remark}

We finish this section with the following lemma, which will be used in \secref{SScob}.

\begin{Lemma}\label{LXT}
Let $(X,T)$ be a ($k-1$)-connected pair of CW-spaces of finite $k$-type ($k\geq 2$). Suppose that 
$X$ and $T$ are \cs in degree $k$. Then, $H_k(\tilde T,\tilde X)$ is a finitely generated projective 
$\bbz\pi$-module ($\pi=\pi_1(X)\approx\pi_1(T)$) and
$$
w_k(T) = w_k(X) + [H_k(\tilde T,\tilde X)] \, .
$$
\end{Lemma}

\begin{proof}
We suppose that $(X,T)$ is a pair of CW-complexes with finite $k$-skeleta.
Let $K=X^{k-1}$. Then $(X,K)$ and $(T,K)$ are ($k-1$)-connected and,
by \eqref{defwk}, one has
$$
w_k(X) = [H_k(\tilde X,\tilde K)]   \ \hbox{ and } \   w_k(T) = [H_k(\tilde T,\tilde K)] \, .
$$
Let us consider the homology exact sequence of the triple $(\tilde T,\tilde X,\tilde K)$. 
One has $H_{k-1}(\tilde K,\tilde X)=0$ since $(X,K)$ is ($k-1$)-connected. The connecting homomorphism
$H_{k+1}(\tilde T,\tilde X)\to H_k(\tilde X,\tilde K)$ factors through $H_k(\tilde X)$ which vanishes 
by \lemref{LW1}. We thus get a short exact sequence
$$
0 \to H_k(\tilde X,\tilde K) \to H_k(\tilde T,\tilde K) \fl{\beta} H_k(\tilde T,\tilde X) \to 0  \ .
$$
We claim that this sequence splits, which will prove the lemma. To construct a section $\sigma$ of $\beta$, 
let us consider the commutative diagram
\beq{LcobL1-dia}
\dia{ 
H_k(\tilde T,\tilde K) \ar[d]_\beta \ar@{>->}[r] & H_{k-1}(\tilde K) \ar[d]_{\tilde\varphi_*} \ar[r] & 
H_{k-1}(\tilde T)  \ar@{<->}[d]_=
\\  
H_k(\tilde T,\tilde X) \ar@/_3mm/@{..>}[u]_{\sigma} \ar@{>->}[r] & H_{k-1}(\tilde X) \ar[r] \ar@/_3mm/[u]_{\sigma_0}
& H_{k-1}(\tilde T)
}
\eeq
in which the horizontal lines are exact. The left hand horizontal arrows are injective since
$H_k(\tilde T)=0$ by \lemref{LW1}. A seen in the proof of \lemref{WL2b}, the inclusion $\varphi\:K\hookrightarrow X$,
considered as a Serre fibration, admits a section over $X^k$, whence the section $\sigma_0$ of 
$\tilde\varphi_*$. A chasing argument in Diagram~\eqref{LcobL1-dia} then produces the 
required section $\sigma$ of $\beta$.  
\end{proof}

\subsection{Realization of cell-dispensability obstructions}

We now give two propositions concerning the realization of elements of $\tilde K_0(\bbz\pi)$ as the obstruction 
$w_k(X)$ for some cohomology $k$-silent complex $X$ with an identification $\pi_1(X)\approx\pi$. This identification
will be achieved via a reference map to a fixed finite cell complex $Y$ (see \ref{reference}).

\begin{Proposition}\label{PW3}
Let $Y$ be a connected finite CW-complex of dimension \\ $k-1\geq 2$ and let $\ell> k$ be an integer.
Let $\calp\in\tilde K_0(\bbz\pi_1(Y))$. 
Then there exists a $Y$-referred finite CW-complex $(X_k^\ell,g^\ell)$ of dimension $\ell+1$ such that 
$X_k^\ell$ is \ccs{k}{\ell} and satisfies $g^\ell_*(w_k(X))=\calp$.
\end{Proposition}

\begin{proof}
Let us represent the class $(-1)^k\calp$ by the image $P$ of a projector ${\rm pr}_P\:F\to F$, where $F$ is a 
finitely generated free $\bbz\pi_1(Y)$-module. Then $F$ admits a supplementary projector
${\rm pr}_Q\:F\to F$ with image $Q=\ker {\rm pr}_P$, and $F=P\oplus Q$. 
Let $\calb=\{b_1,\dots,b_r\}$ be a basis of $F$. We introduce some precise notations which will also be used
in the proof of \proref{PW4}. Let 
\beq{PW3e15}
\cals = S^{k-1}_1 \vee \cdots \vee S^{k-1}_r \quad , \quad  \cald = D^{k}_1 \amalg \cdots \amalg D^{k}_r
\eeq
where $S^{k-1}_j$ and $D^k_j$ are copies of the sphere $S^{k-1}$ and of the disk $D^k$. 

We shall add cells to $Y$ to get skeleta $X^j$ of a cell complex $X$, together with a $Y$-reference retraction
$g^j\:X^j\to Y$, such that $(X^\ell,g^\ell)=(X_k^\ell,g^\ell)$ when $\ell>k$
(recall from \lemref{Lrefretr} that an $Y$-reference map is anyway, up to homotopy equivalence, equal to a retraction).

The $(k-1)$-skeleton of $X$ is defined to be 
$$
X^{k-1}=Y\vee \cals \, , 
$$
with the $Y$-reference retraction $g^{k-1}\:X^{k-1}\to Y$ sending $\cals$ to the base point.

Since $k\geq 3$, $\pi=\pi_1(Y)\approx\pi_1(X^{k-1})$ and
$\pi_{k-1}(X^{k-1},Y)\approx H_{k-1}(\tilde X^{k-1},\tilde Y)$ is identified with $F$ (the inclusion 
$S^{k-1}_j\to\cals\to X^{k-1}$ representing the element $b_j\in F$). Since $g^{k-1}$ is a retraction, the
homotopy exact sequence of the pair
$(X^{k-1},Y)$ splits
$$
\dia{ 
0 \ar[r] & \pi_{k-1}(Y) \ar[r] & \pi_{k-1}(X^{k-1}) \ar@/^2mm/@{>>}[l]^{g^{k-1}_*}   \ar[r] & \pi_{k-1}(X^{k-1},Y) \ar@/^2mm/@{.>}[l]^{\sigma}
\ar[r] & 0
} \, .
$$
Let 
\beq{sigma}
\sigma :\pi_{k-1}(X^{k-1},Y)\to\pi_{k-1}(X^{k-1})
\eeq
be the section sending $\pi_{k-1}(X^{k-1},Y)$ isomorphically onto
$\ker g^{k-1}_*$. 
Form the cell complex 
$$
X^k = X^{k-1} \cup_\theta \cald  \, .
$$
where the restriction to $\bd(D_j^k)$ of the attaching map $\theta\:\bd (\cald)\to X^{k-1}$  
represents $\sigma\pcirc {\rm pr}_P(b_j)\in\pi_{k-1}(X^{k-1})$. As $g^{k-1}_*\pcirc\sigma=0$, 
the retraction $g^{k-1}$ extends to a retraction $g^k\:X^k\to Y$. 
Again, $\pi_{k}(X^k,X^{k-1})\approx H_{k}(\tilde X^k,\tilde X^{k-1})$ is identified with $F$. One has 
a diagram whose horizontal lines are exact:
$$
\dia{
& \pi_k(Y) & \\
\pi_k(X^{k-1}) \ar[rr] \ar@{>>}[ur]^{g^{k-1}_*} && 
\pi_{k}(X^k) \ar[r] \ar@{>>}[ul]_{g^k_*}  &  \pi_{k}(X^k,X^{k-1})   \ar[r] &  \pi_{k-1}(X^{k-1}) \\
 && F   \ar@{.>}[u]_\tau   \ar[r]^{{\rm pr}_Q} & F \ar[u]_{\approx} \ar[r]^{{\rm pr}_P} & F \ar[u]_{\sigma}
}  \ .
$$
The right square being commutative, there exist homomorphisms $\tau\: Q\to \pi_{k}(X^k)$
making the left square commutative. Since $g^{k-1}_*$ is onto, there is such a $\tau$ such that $g^k_*\pcirc\tau=0$. 
Form the cell complex 
$$
X^{k+1} = X^k \cup e_1^{k+1} \cup\cdots\cup e_r^{k+1} \, ,
$$
where the $(k+1)$-cell $e_j^{k+1}$ is attached by the element $\tau\pcirc {\rm pr}_Q(b_j)\in\pi_k(X^k)$.
Since $g^k_*\pcirc\tau=0$, the retraction $g^k$ extends to a retraction $g^{k+1}\:X^{k+1}\to Y$. 
One checks that $H_k(\tilde X^{k+1})=0$ (one had $H_k(\tilde X^{k})\approx Q$) and that 
\beq{PW3-e5}
C_{k-1}(\tilde X^{k+1})\approx C_{k-1}(Y)\oplus Q\oplus B_{k-1}(\tilde X^k) \approx C_{k-1}(Y)\oplus Q\oplus P \, .
\eeq
By \lemref{LW1}, the cell complex $X^{k+1}$ is \cs in degree $k$ and $g^{k+1}_*(w_k(X))=\calp$ by \remref{WR1}.2.
Thus, for $\ell=k+1$, we can just define $X_k^\ell = X^{k+1}$. 

Now, since $B_k(\tilde X^{k+1})\approx Q$ is projective, one has isomorphisms of $\bbz\pi$-modules
\beq{PW4e1}
\dia{
C_{k+1}(\tilde X^{k+1}) \ar[r]^(0.37){\approx} \ar[d]^\approx  & B_{k}(\tilde X^{k+1})\oplus Z_{k+1}(\tilde X^{k+1}) 
\ar[d]^\approx     \\
F  \ar[r]^\approx &  Q\oplus P
}
\eeq
Exactly like above, on may attach $r$ cells of dimension $k+2$ to $X^{k+1}$ to obtain $X^{k+2}$
and the retraction $g^{k+2}$ onto $Y$, so that the boundary homomorphism 
$$
\partial :  F \approx C_{k+2}(\tilde X^{k+2})\to C_{k+1}(\tilde X^{k+1})\approx Q\oplus P
$$ 
coincides with ${\rm pr}_P$.  Therefore, 
$H^{k+1}(\tilde X^{k+2})=0$. Together with \eqref{PW4e1},
this implies by \lemref{LW1} that $X^{k+2}$ is \cs in degree $k+1$. 

The process may be repeated  one degree higher, using the diagram like \eqref{PW4e1} with $k$ replaced by $k+1$
and $P$ and $Q$ being exchanged. Carrying on this procedure, one gets the $Y$-referred pair $(X_k^\ell,g^\ell)$ for $\ell>k$ enjoying the required properties.
\end{proof}

\begin{Remark}\label{RW3} \rm
The pair $(X_k^\ell,Y)$ as constructed in the above proof is weakly ($k-1$)-connected but {\it not} 
($k-1$)-connected. Indeed, the retraction $g^\ell$ splits the homology exact sequence of $(\tilde X_k^\ell,\tilde Y)$
$$
0 \to H_j(\tilde Y) \to  H_j(\tilde X_k^\ell) \to  H_j(\tilde X_k^\ell,\tilde Y) \to 0  
$$
and $H_j(\tilde X_k^\ell)=0$ for $k\leq j < \ell$ by \lemref{LW1}. Therefore, if $H_{k-1}(\tilde X_k^\ell,\tilde Y)=0$,
the connectivity of the pair $(X_k^\ell,Y)$ would be $\ell-1\geq k$ and thus $w_k(X_k^\ell)=0$. 
\mancqfd
\mancqfd\end{Remark}

Passing to the limit for $\ell\to\infty$, \proref{PW3} and its proof give rise to the following result which 
will be used in \secref{Srea}.

\begin{Proposition}\label{PW4}
Let $Y$ be a connected finite CW-complex of dimension \\ $k-1\geq 2$.
Let $\calp\in\tilde K_0(\bbz\pi_1(Y))$. Then there exists an $Y$-referred space $(X_k,g)$ such that 
\begin{itemize}
\item[(a)] $X_k$ is of finite type.
\item[(b)] $X_k$ is \cs in degree $j$ for all integers $j\geq k$ and satisfies $g_*(w_k(X_k))=\calp$. 
\item[(c)] $X_k$ is homotopy equivalent to a countable CW-complex $L$ is of dimension~$k$.
\item[(d)] For $\ell>k$, the inclusion $X_k^\ell\hookrightarrow X_k$ is a domination. 
In particular, $X_k$ and $L$ are finitely dominated.
\end{itemize}
\end{Proposition}

\begin{proof}
Define $X_k$ as the union of $X_k^\ell$ for $\ell\geq k$, endowed with the weak topology 
(note that $X_k^\ell\subset X_k^{\ell+1}$). The retractions $g^\ell\:X_k^\ell\to Y$ induce the required $Y$-reference
retraction $g\:X_k\to Y$ and the pair $(X_k,g)$ clearly satisfies (a) and~(b).

To prove (d), we replace up to homotopy equivalence the inclusion $\psi^\ell\: X_k^\ell\hookrightarrow X_k$ 
by a Serre fibration $\hat\psi^\ell\:\hat X_k^\ell\to X_k$. 
Its fiber $F$ being ($\ell-1$)-connected, the obstructions to get a section of $\hat\psi^\ell$ belong to $H^{i}(X_k;\pi_{i-1}(F))$
for $i>\ell$.  If $\ell>k$, the latter vanishes by (b). The map $\psi^\ell$ thus admits a section and is therefore a domination. 

It remains to prove (c). We use the notations of the proof of \proref{PW3}: $F$, $\calb$, 
${\rm pr}_P,{\rm pr}_Q\:F\to F$, $\cals$, $\cald$, $\sigma$, etc. Let $K=X_k$. Recall from \eqref{PW3-e5} that $K^{k-1}=Y\vee\cals$, so that
\beq{ck-1K}
C_{k-1}(\tilde K)\approx C_{k-1}(Y)\oplus Q\oplus P
\eeq
with $P=B_{k-1}(\tilde K)$ and $Q=H_{k-1}(\tilde K)$.
The ($k-1$)-skeleton of $L$ will be 
$$
L^{k-1} =  
K^{k-1} \vee \cals_1 \vee\cals_2 \vee \cdots = Y \vee \cals_0 \vee \cals_1 \vee \cals_2 \vee \cdots
$$
where $\cals_i$ is a copy of the bouquet $\cals$. Therefore,
\begin{eqnarray}
C_{k-1}(\tilde L) &\approx & 
C_{k-1}(\tilde K) \oplus F \oplus F \oplus \cdots \nonumber  \\ &\approx &
C_{k-1}(\tilde Y) \oplus Q \oplus P\ \oplus F \oplus F \oplus \cdots  \nonumber  \\ &\approx &
C_{k-1}(\tilde Y) \oplus Q \oplus P \oplus (Q\oplus P) \oplus (Q\oplus P)  \oplus \cdots  \label{PW4e2b}
\\ &\approx &
C_{k-1}(\tilde Y) \oplus Q \oplus (P\oplus Q) \oplus (P\oplus Q) \oplus \cdots \label{PW4e2}
\end{eqnarray}
The CW-complex $L$ is then obtained by attaching $k$-cells to $L^{k-1}$:
\beq{edefLk}
L = L^{k-1} \cup_{\theta_{0,1}} \cald_{0,1} \cup_{\theta_{1,2}} \cald_{1,2} \cup \cdots \, ,
\eeq
where $\cald_{r,s}$ is a copy of $\cald$, with the attaching map $\theta_{r,s}\:\bd (\cald_{r,s})\to L^{k-1}$
which we will now define.
One has $\pi_{k-1}(L^{k-1},Y)\approx F_0\oplus F_1\oplus\cdots$,
where $F_j$ is a copy of $F$. Since the pair $(L^{k-1},Y)$ is weakly ($k-2$)-connected, 
the homomorphism $\pi_{k-1}(L^{k-1})\to \pi_{k-1}(L^{k-1},Y)$ is onto and one can 
choose a section 
$\sigma_L\:\pi_{k-1}(L^{k-1},Y) \to \pi_{k-1}(L^{k-1})$ whose restriction to $F_j$ is denoted by $\sigma_j$.
For $\sigma_0$, we make our choice so that the following diagram
$$
\dia{
F \ar[r]^(0.28){\approx} \ar[d]^(0.44)= & \pi_{k-1}(K^{k-1},Y) \ar[r]^(0.52)\sigma \ar[d] & \pi_{k-1}(K^{k-1}) \ar[d]
\\
F_0 \ar@{>->}[r] & \pi_{k-1}(L^{k-1},Y) \ar[r]^(0.52){\sigma_0} & \pi_{k-1}(L^{k-1}) 
}
$$
is commutative, where the vertical arrows are induced by the inclusion $K^{k-1}\to L^{k-1}$ and $\sigma$ is the section 
of \eqref{sigma}. The attaching map $\theta_{r,s}\:\bd (\cald_{r,s})\to L^{k-1}$ is now chosen
so that its restriction to $\bd(D_j^k)$
represents the class $\sigma_r\pcirc {\rm pr}_P(b_j)+ \sigma_s\pcirc {\rm pr}_Q(b_j) \in\pi_{k-1}(L^{k-1})$. 
One can thus write
\beq{PW4e4}
C_{k}(\tilde L) \approx F \oplus F \oplus \cdots \approx
(P\oplus Q) \oplus (P\oplus Q) \oplus \cdots 
\eeq
so that the boundary homomorphism
$\partial\: C_{k}(\tilde L)\to C_{k-1}(\tilde L)$ sends the 
$j$-th term $(P+Q)$ of \eqref{PW4e4} isomorphically, 
through ${\rm pr}_P\oplus {\rm pr}_Q$,
to the corresponding summand of \eqref{PW4e2}. This boundary homomorphism
being injective, one has $H_k(\tilde L)=0$. 

The above definitions of $K$ and $L$ guaranty the existence of a map 
$f\: L\to K=X_k$ extending the obvious identification $Y\vee\cals_0\fl{\simeq}Y\vee\cals$ and, for $j\geq 1$, sending
$\cals_j$ and $\cald_{j,j+1}$ to the base point of $K$. The pairs $(X_k,Y)$ and $(L,Y)$ being weakly ($k-1$)-connected,
so is the map $f$. The homomorphism $C_{k-1}\tilde f\: C_{k-1}(\tilde L)\to C_{k-1}(\tilde K)$
sends the first summands $C_{k-1}(\tilde Y)\oplus Q$ of \eqref{PW4e2b} isomorphically those of \eqref{ck-1K}. 
But these summands provide isomorphisms $H_{k-1}(\tilde Y)\oplus Q\fl{\approx} H_{k-1}(\tilde L)$ and 
$H_{k-1}(\tilde Y)\oplus Q\fl{\approx} H_{k-1}(\tilde K)$. Therefore,
$H_{k-1}\tilde f \: H_{k-1}(\tilde L) \to H_{k-1}(\tilde K)$ is an isomorphism.
Since $H_j(\tilde L)=0=H_j(\tilde K)$ for $j\geq k$, the map $f$ is a homotopy equivalence.
\end{proof}

Here are a few remarks about \proref{PW4} and its proof. 

\begin{Remark}\label{remcellfree} \rm
When $\calp\neq 0$, The complex $X_k$ of \proref{PW4} is of finite type and is homotopy equivalent to a complex  without $j$-cells for $k\leq j\leq\ell$ (but not simultanously). So $X_k$ is not \cef{k}{\ell}.
\mancqfd\end{Remark}

\begin{Remark}\label{WR7} \rm
The inclusion $\beta\:K^{k-1}\hookrightarrow L$  is ($k-1$)-connected and $K^{k-1}$ is a finite complex. 
Since $H_k(\tilde L)=0$, one has 
$$
\dia{
0 \ar[r] & H_k(\tilde\beta) \ar[r] & H_{k-1}(\tilde K^{k-1}) \ar[r] & H_{k-1}(\tilde L) \ar[r] & 0 \\
&& H_{k-1}(\tilde Y)\oplus F \ar[u]_\approx 
\ar[r]^{\id\oplus{\rm pr}_Q} & H_{k-1}(\tilde Y)\oplus Q \ar[u]_\approx \ar[r] & 0
}\ .
$$
Therefore, $\pi_k(\beta)\approx H_k(\tilde\beta)\approx P$ and, in $\tilde K_0(\bbz\pi)$, one gets
$$
w_k(L) = (1)^k[\pi_k(\beta)] = (1)^k[P] = \calp
$$ 
as expected. 
\mancqfd\end{Remark}

\begin{Remark}\label{RWallObs}\rm
The CW-complex $L$ (or $X_k$) fulfills the conditions for the Wall finiteness obstruction $\wa(L)\in\tilde K_0(\bbz\pi_1(L))$
to be defined \cite[Theorem F]{WallFin}. Using \remref{WR7}, the obstruction $\wa(L)$ is equal to $\calp$, thus equal to $w_k(L)$, as claimed in \remref{WR1}.1. 
\mancqfd\end{Remark}

\begin{Remark}\label{WR10}\rm
The CW-complex $L$ may be obtained as the direct limit of a nested system
\beq{WR10-dia}
Y  \hookrightarrow L_1 \hookrightarrow L_2  \hookrightarrow \cdots \hookrightarrow
L_\infty = L \fl{g} Y
\eeq
where
$$
L_1= Y\vee\cals  \quad {\rm and } \quad L_{j+1} = L_j \vee \cals \cup \cald \ \ {\rm for}\ j\geq 1 
$$
and such that the composition of maps in \eqref{WR10-dia} is the identity of $Y$. Indeed, the inclusion chain reflects 
the very definition of $L$. 
Let $f\:L\to X_k$ be the homotopy equivalence constructed to prove Part (c) of \proref{PW4}. Then,
the composition of $f$ with the $Y$-reference map $g\: X_k\to Y$ gives 
an $Y$-reference retraction from $L$ onto $Y$, also called $g\:L\to Y$. 
Its restriction to $L_j$ provides an $Y$-reference retraction $g_j\: L_j\to Y$. 
This material is bound to be used in \secref{Srea}.
\mancqfd\end{Remark}

\subsection{Relationship with Wall's finiteness obstruction}\label{SWallFO}

A certain relationship between the  cell-dispensability obstruction $w_k(X)$ and the
Wall's finiteness obstruction was already noticed in \remref{WR1}. The following variant of \proref{PW4}  may shed more light about this matter.

\begin{Proposition}\label{Pw-wall}
Let $X$ be a finite CW-complex of dimension $\ell$ which is \ccs{k}{\ell}. 
Then, there is a finitely dominated CW-complex $X'$ containing $X$ as a subcomplex, such that $(X',X)$ is $\ell$-connected and $X'$ is homotopy equivalent to a countable CW-complex of dimension $k$. Two such complexes 
$X'_1$ and $X'_2$ are homotopy equivalent relative $X$. Moreover, the image of $w_k(X)$ into $\tilde K_0(X')$ coincides with $\wa(X')$. 
\end{Proposition}

\begin{proof}
Starting with with $X$ we apply the process of the proof of Propositions~\ref{PW3} and~\ref{PW4}, ending with 
the space $X'= X_k$ of \proref{PW4}, which enjoys the required properties. If $X'_1$ and $X'_2$
are two such complexes, as $X'_1$ is \cs in degree $\geq k$, there is no obstruction to extend ${\rm id}_X$ to a map $f\:X'_1\to X'_2$, which will be a homotopy equivalence, since $H_j(\tilde{X'_1})=0=H_j(\tilde{X'_2})$
for $j\geq k$ (by \lemref{LW1}). Finally, the complex $X'$ fulfills the conditions for the Wall finiteness obstruction 
$\wa(X')\in\tilde K_0(\bbz\pi_1(X'))$ to be defined \cite[Theorem F]{WallFin} and the homological algebra definitions 
of $w_k(X)$ and of $\wa(X')$ make them to be equal modulo the identification $\pi_1(X)\approx\pi_1(X')$.  
 \end{proof}

Another consequence of the formally identical  homological algebra definitions 
of $w_k(X)$ and of $\wa(X')$ is the following product formula. For the Wall obstruction, this is a particular
case of the product formula proven in Siebenmann's thesis \cite[Theorem~7.2]{SieTh} (see also \cite[Theorem~3.14, p.~169]{Vara}).

\begin{Proposition}[Product formula]\label{PproductFor}
Let $X$ be a connected CW-space of finite type which is \ccs{k}{\ell} for $k\geq 3$. Let $A$ be a 
connected finite CW-complex of dimension $a\leq \ell-k$. Then $X\times A$ is \ccs{k+a}{\ell}
and the equality
$$
w_{k+a}(X\times A)= \chi(A)\, i_*(w_{k}(X))
$$
holds in $\tilde K_0(\bbz(\pi_1(X\times A)))$, where $\chi(A)$ is the Euler characteristic of $A$ 
and $i_*\:\tilde K_0(\bbz\pi_1(X))\to\tilde K_0(\bbz\pi_1(X\times A))$ 
is induced by the inclusion\\ $i\:\pi_1(X)\to\pi_1(X\times A)\approx \pi_1(X)\times\pi_1(A)$ sending $u\in\pi_1(X)$
to $(u,1)$.
\end{Proposition}

\begin{proof}
The homological algebra's arguments used in \cite[pp.~168--170]{Vara} to prove the same formula for the Wall 
finiteness obstruction work as well to establish the proposition.
\end{proof}

\begin{Corollary}
 Let $X$ be a connected CW-space of finite type which is \ccs{k}{\ell} for $k\geq 3$.
 Then $X\times S^1$ is \cef{k+1}{\ell}. \mancqfd
\end{Corollary}

For more applications of \proref{PproductFor}, see \proref{CCproductFor}

\section{Antisimplicity}\label{GScasm}

\subsection{Preliminaries}

\begin{ccote}\label{holoc} Homology with local coefficients. \rm 
While cohomology with local coefficients was defined for an arbitrary pair (see~\ref{cohmod}),
homology with local coefficients is defined for a {\it pair over $\rpi$}. A \dfn{space over $\rpi$} is a 
space $X$ together with a (homotopy class of) map $w_X\:X\to\rpi$. A pair $(X,Z)$ is said to be 
\dfn{over $\rpi$} if both $X$ and $Z$ are over $\rpi$ with $\omega_Z$ being the restriction of 
$\omega_X$ to $Z$. 

For example, the space $BO$ is over $\rpi$ by the classifying map
$$
\omega_{BO} : BO \to B{\pi_1(BO)} \simeq B\{\pm 1\} \simeq \rpi \, .
$$
Thus, a stable vector bundle $\eta$ over $X$, classified by $\Phi_\eta\:X\to BO$, makes $X$ a space over $\rpi$
by the map  $\omega_\eta=\omega_{BO}\pcirc \Phi_\eta$, called the \dfn{orientation character of~$\eta$}.
In the case of a (smooth) manifold $M$, the orientation character of its stable normal (or tangent) bundle 
gives the \dfn{orientation character $\omega_M\:M\to\rpi$ of~$M$}. 

When $X$ is a CW-space, the functor $\pi_1$ provides a bijection
$$
[X,\rpi] \fl{\approx}  \hom\big(\pi_1(X),\{\pm 1\}\big)   \ , \  (f\mapsto \pi_1f) \ .    
$$
We use the same notation $\omega_X$ for 
the map $\omega_X$ itself and for its induced homomorphism $\pi_1\omega_X$. 

A group homomorphism $\omega\:\pi\to \{\pm 1\}$ produces an involution $a\mapsto\bar a$ on $\bbz\pi$
defined by 
$$
\overline{\sum_{g\in\pi} n_g g} = \sum_{g\in\pi} \omega(g) n_g g^\mun \ ,
$$
satisfying $\overline{a+b}=\bar a + \bar b$ and $\overline{ab}=\bar b\bar a$. Any left $\bbz\pi$-module $R$, 
may thus be transformed into a right $\bbz\pi$-module $\bar R$ with the action $r a=\bar a r$ for $a\in\bbz\pi$.
Let $(X,Z)$ be a pair of CW-spaces over $\rpi$. Set $(\pi,\omega) = (\pi_1(X),\omega_X)$ and let $R$ be a (left)
$\bbz\pi$-module.
The homology $H_*(X,\bd Z;R)$ is defined as the homology of the 
chain complex 
$$
\overline{C_*(\tilde X,\tilde Z)}\otimes_{\bbz\pi} R \, ,
$$
(see \cite[p.~15]{HVBoo}). 
For example, let $(M,\bd M)$ be a compact manifold pair, with $M$ connected.
When $R=\bbz$ with trivial $\pi_1(M)$-action,
then $H_*(M,\bd M;\bbz)$ is the ordinary homology of $(M,\bd M)$ with integral coefficients when $M$ is orientable,
and that of the orientation cover $(M^{or},\bd M^{or})$  otherwise.
In any case, $H_n(M,\bd M;\bbz)\approx\bbz$. A compact connected manifold of dimension $n$ is always supposed
to be equipped with a generator $[M]$ of $H_n(M,\bd M;\bbz)\approx\bbz$, called the \dfn{fundamental class} of $M$. 
\end{ccote}

\begin{ccote}\label{Pesp} \pe spaces. \rm 
Let $P$ be a finitely dominated CW-space over $\rpi$, together with a class $[P]\in H_n(P;\bbz))$.
We say that $P$ is a \dfn{\pe space of formal dimension $n$} if the the cap product with $[P]$ 
\beq{PeDef1}
 - \frown [P] : H^r(P,Z;R) \fl{\approx}  H_{n-r}(P;R)
\eeq
is an isomorphism for any $\bbz\pi_1(P)$-module $R$.
More generally, let $(P,Q)$ be a pair over $\rpi$ of finitely dominated CW-spaces, together with a  class 
$[P]\in H_n(P,Q))$. We say that $(P,Q)$ is a \dfn{\pe pair of formal dimension $n$} if the following
two conditions hold true 
\begin{itemize}
\item[(a)] the cap product with $[P]$ 
\beq{PeDef1b}
 - \frown [P] : H^r(P,Q;R) \fl{\approx}  H_{n-r}(P;R)
\eeq
is an isomorphism for any $\bbz\pi_1(P)$-module $R$, and
\item[(b)] the space $Q$ together with the class $[Q]=\partial_*([P])$ is a \pe space of formal dimension $n-1$.
 \end{itemize}
 We may also say that $P$ is a \dfn{\pe space} of formal dimension $n$, {\it with boundary $\bd P=Q$}. 
This boundary is then a \pe space without boundary. 

Finally, by a \dfn{\pe cobordism} $(P,Q_1,Q_2)$, we mean a pair $(P,Q)$ of  
finitely dominated CW-spaces with $Q=Q_1\dcup Q_2$, together with a class 
$[P]\in H_n(P,Q)$ such that (a) above is satisfied and such that $Q_i$ is a \pe space with fundamental class 
$[Q_i]$ such that $\partial_*([P])=[Q_1]-[Q_2]$. 

\begin{Remarks}\rm \ 
(1) A compact $n$-manifold pair $(M,\bd M)$ is a \pe space of formal dimension $n$. This follows from the 
standard proof of the \pe-Lefschetz duality using a $\calc^\infty$-triangulation of $M$ 
and its dual cell decomposition.

\nsk{1}
(2) By \cite[Theorem~B]{KleinQin}, Condition (a) is equivalent to

\nsk{2}\hskip 1.8mm\begin{minipage}{10cm}
\begin{itemize}
\item[(a')] the cap product with $[P]$ 
\beq{PeDef2}
 - \frown [P] : H^r(P;R) \fl{\approx}  H_{n-r}(P,Q;R)
\eeq
is an isomorphism for any $\bbz\pi_1(P)$-module $R$.
\end{itemize}
\end{minipage}

Hence, if (a) and/or (a') is true, one has a sign-commutative diagram 
for the (co)homology in a $\bbz\pi_1(P)$-module $R$: 

\sk{1}
\begin{minipage}{80mm}
\footnotesize \hskip -6mm
$
\dia{
H^{j}(P,Q) \ar[r] \ar[d]^{\frown [P]}_\approx & H^{j}(P) \ar[r] \ar[d]^{\frown [P]}_\approx 
& H^j(Q) \ar[r] \ar[d]^{\frown [Q]}  & H^{j+1}(P,Q) \ar[r] \ar[d]^{\frown [P]}_\approx  & H^{j+1}(P)  
\ar[d]^{\frown [P]}_\approx 
\\  
H_{n-j}(P) \ar[r] & H_{n-j}(P,Q) \ar[r] & H_{n-1-j}(Q) \ar[r] &  H_{n-1-j}(P) \ar[r] & H_{n-1-j}(P,Q)
}
$
\end{minipage}\rm\nsk{2}
By the five lemma, $-\frown [Q]\: H^j(Q;R) \to H_{n-1-j}(Q;R)$ is an isomorphism for any $\bbz\pi_1(P)$-module $R$,
but this may be not the case if $R$ is a $\bbz\pi_1(Q)$-module (see \cite[Theorem~A]{KleinQin}). Thus, Conditions (a) and (b) are independent in general. Note however, that (a) implies (b) when $\pi_1(Q)\approx\pi_1(P)$. 
\mancqfd
\end{Remarks}
\end{ccote}

\begin{ccote}\label{coSpivak} Spivak fibrations. \ \rm
Let $BG$ be the classifying space for stable spherical fibration. 
Let $(X,Y)$ be a pair over $BG$, i.e. a pair $(X,Y)$ together
with a map $g\:X\to BG$. Consider the graded abelian group
$$
\calh_*(X,Y) = \pi_n^s(T(\xi),T(\xi|Q))   \quad , \quad T(\, )  = \hbox{ Thom spectrum } \, ,
$$
where $\xi$ is the pull-back by $g$ of the universal stable spherical fibration over $BG$. 
Then $\calh_*$ is a homology theory for spaces over $BG$. 

The natural maps $BG\to\rpi$
makes a space over $BG$ a space over $\rpi$. Therefore, the homology $H_*(X,Y;R)$ with coefficient in a
$\bbz\pi_1(X)$-module $R$, in the sense of~\ref{holoc}, is defined. When $R=\bbz$ with trivial $\pi_1(X)$-action,
one has degree~$0$ graded homorphism 
\beq{deft}
t : \calh_*(X,Y) \to H_*(X,Y;\bbz)
\eeq
obtained by the composition of the Hurewicz map with the Thom isomorphism. 

Let $(P,Q)$ be a \pe pair of formal dimension~$n$. 
A \dfn{Spivak fibration} for $(P,Q)$ is a stable spherical fibration $\xi$ over $P$, classified by a lifting
$P\to BG$ of $\omega_P$, such that there exists a class $(P) \in \calh_n(X,Y)$
satisfying $t((P))=[P]$.
We call $(P)$ a \dfn{Spivak class}. In our context of (finitely dominated) \pe spaces,
Spivak fibrations do exist and, if $\xi_j$ ($j=1,2$) are two
such fibrations with Spivak classes $(P_j)$, there is a fibre homotopy equivalence from $\xi_1$ to $\xi_2$
sending $(P_1)$ to $(P_2)$ \cite[Corollary~3.6]{WallPoinc}. 

\begin{Remark}  \rm \
Although our definitions are inspired by those of \cite{HVBoo}, there are some differences. What we call a  \pe space
is called a {\it PD-space} there. A {\it \pe} space in \cite{HVBoo} is a PD-space thgether with a Spivak class.
The present paper is self-contained with respect to \cite{HVBoo}.
\mancqfd\end{Remark}
\end{ccote}

\begin{ccote}\label{MPfnormal} $\eta$-normal maps.    \ \rm
Let $Y$ be a connected CW-space and let $\eta$ be a stable vector bundle over $Y$, with characteristic map
$\Phi_\eta\: Y\to BO$.

\begin{Definitions}\rm 
(A) A map $q\:Z\to Y$, from a manifold $Z$ to $Y$ is called \dfn{$\eta$-normal} if 
$\Phi_\eta\pcirc q$ classifies the stable normal bundle $\nu_Z$ of $Z$.

\nsk{1}
(B) If $(P,Q)$ is a \pe pair, 
a map $q\:P\to Y$ is called \dfn{$\eta$-normal} if the composite
$$
P \fl{q} Y \fl{\Phi_\eta} BO \to BG
$$
classifies a Spivak fibration for $(P,Q)$.
\end{Definitions}

The following result is well known.

\begin{Lemma}\label{Lnormal}
(1) A map from a compact manifold $M$ to $Y$ which is $\eta$-normal in the sense of (A) 
is $\eta$-normal in the sense of (B) for the pair $(M,\bd M)$.
\nsk{1}
(2) Let $f\:(P',Q')\to (P,Q)$ be a map of degree 1 (i.e. $f_*([P'])=[P]$) between \pe pairs and let 
$q\:P\to Y$ be a map. If the map $q\pcirc f$ is $\eta$-normal, so is the map $q$.
\nsk{1}
(3) If $q\: P\to Y$ is $\eta$-normal for the \pe pair $(P,Q)$, then its restriction to the \pe space $Q$ is 
$\eta$-normal. 
\end{Lemma}

\begin{proof}
The Thom-Pontryagin class $(M)_{TP}$ is a Spivak class for $(M,\bd M)$, which proves (1). For (2), we note that the class $f_*((P))$ is a Spivak class for $(P,Q)$. For (3), we use that the restriction over $Q$ of a Spivak fibration for $(P,Q)$ is a Spivak fibration for $Q$, taking $(Q)=\partial_*((P))$.
\end{proof}

In most of our applications, $Y$ will be a finite complex of dimension $k-1$ and $q$ will be an $Y$-reference map.
If $q$ is $\eta$-normal, we call $q$ an \dfn{$(Y,\eta)$-reference map}.

Let $(P,Q)$ be a \pe pair of formal dimension $n$ and let $q\:P\to Y$ be an $\eta$-normal map. 
The Spivak fibration $\xi$ over $(P,Q)$ is thus endowed with the vector bundle reduction $q^*\eta$.
This corresponds to a surgery problem (normal map of degree one) $\alpha\:(\calm,\bd\calm)\to (P,Q)$, where $\calm$ is a compact $n$-manifold and $q\pcirc\alpha$ is $\eta$-normal.
Such a data determines a surgery obstruction in the Wall group $L_n^p(\pi_1(P),\pi_1(Q),\omega_{P})$,
vanishing if and only if $\alpha$ is normally cobordant to a homotopy equivalence. When $Q=\emptyset$,
We denote by $\sigma(P,q)\in L_n^p(\pi_1(Y),\omega_{\eta})$ the image by $q^*$ of this surgery obstruction.
\end{ccote}

\begin{ccote}\label{PdecNor} $\eta$-normal \pe decompositions. \ \rm
Let $(X_\pm,\bd X_\pm)$ be two CW-pairs with $\bd X_-=Z_-\dcup Z_0$ and $\bd X_+ = Z_0\dcup Z_+$. Form the CW-pair $(X,\bd X)$ where
\beq{decompEx}
X= X_- \cup_{Z_0} X_+  \ , \ \bd X = Z_- \dcup Z_+ \, .
\eeq
For the sake of simplicity, we assume that the homomorphisms $\pi_1(\bd X_\pm)\to \pi_1(X_\pm)$
and $\pi_1(Z_0)\to\pi_1(X_\pm)$ induced by the inclusions are all isomorphisms. 

Let $Y$ be a connected CW-space and let $\eta$ be a stable vector bundle over $Y$. Let $q\:X\to Y$ be a map,
restricting to $q_\pm\:X_\pm\to Y$. These maps make $X$, $\bd X$, $X_\pm$, etc, spaces over $BO$, $BG$ and $\rpi$. 
Consider their homology $H_n()$ with coefficients in $\bbz$ with trivial $\bbz\pi_1(Y)$-action. 
Define $\beta_H\:H_n(X,\bd X) \to H_n(X_-,\bd X_-)\oplus H_n(X_+,\bd X_+)$ by the following diagram
\beq{defbetah}
\dia{ 
H_n(X,\bd X) \ar[r] \ar@/_7mm/[rr]^{\beta_H}
& H_*(X;\bd X\dcup Z_0)  \ar@{<-}[r]^(.39)\approx_(.4){{\rm excis.}} &
 H_n(X_-,\bd X_-)\oplus H_n(X_+,\bd X_+)
}
\eeq

\begin{Definition}\label{PdecNoDef}\rm
We say that \eqref{decompEx} is an \dfn{$\eta$-normal \pe decomposition} if $(X,\bd X)$ and $(X_\pm,\bd X_\pm)$ are \pe pairs of formal dimension~$n$ satisfying $\beta_H([X]) = ([X_-],[X_+])$ and such that $q$ and $q_\pm$ are $\eta$-normal.
\end{Definition}

We now give results to deduce from weaker hypotheses that a decomposition like in~\eqref{decompEx} is an 
$\eta$-normal \pe decomposition. By excision, one has the isomorphism
\beq{deciso2}
H_*(\bd X)  \fl{\approx} H_*(\bd X\dcup Z_0,Z_0) \,.
\eeq
Diagram \eqref{defbetah} and isomorphism \eqref{deciso2} may be considered as well 
for the homology theory $\calh_*$ over $BG$. Therefore, the exact sequences of the
triple $(X,\bd X\dcup Z_0,Z_0)$ and the homomorphism  $t$ of~\eqref{deft} give rise to a morphism of exact sequences
\beq{PdecNor-dia}
\dia{
\calh_n(X,\bd X)  \ar[r]^(.33){\beta_\calh} \ar[d]^t & \calh_n(X_-,\bd X_-)\oplus \calh_n(X_+,\bd X_+) 
\ar[r]^(.68){\partial_\calh} \ar[d]^{t\oplus t} & \calh_{n-1}(Z_0) \ar[d]^t
\\ 
H_n(X,\bd X)  \ar@{>->}[r]^(.33){\beta_H}  & H_n(X_-,\bd X_-)\oplus H_n(X_+,\bd X_+) \ar[r]^(.68){\partial_H} & 
H_{n-1}(Z_0)
}
\eeq

\begin{Proposition}\label{PPdecNor}
Let $(X)\in \calh_n(X,\partial X)$ such that $\beta_\calh((X))=((X_-),(X_+))$.

(i) Suppose that $(X_\pm,\partial X_\pm)$ are \pe pairs of formal dimension $n$ with $[X_\pm]=t((X_\pm))$ (in other words, $q_\pm$ is $\eta$-normal and $(X_\pm)$ are Spivak clases). Then $(X,\partial X)$ is a \pe pair of formal dimension $n$ with $[X]=t((X))$. 

(ii) Suppose that $(X,\bd X)$ and $(X_-,\bd X_-)$ are \pe pairs of formal dimension~$n$ with $[X]=t((X))$
and $[X_-]=t((X_-))$. Then $(X_+,\bd X_+)$ is a \pe pair of formal dimension~$n$ with $[X_+]=t((X_+))$.

Consequently, in both cases, \eqref{decompEx} is $\eta$-normal \pe decomposition.
\end{Proposition}

\begin{proof}
The requirements about finite dominations are guaranteed by \cite[Complement~6.6]{SieTh}.
The \pe duality statements  are proven in \cite[Theorem~2.1 and its addendum]{WallPoinc}.
The assertions about the Spivak classes are then deduced from the commutativity of the left square in Diagram~\eqref{PdecNor-dia}, using for~(i) that $\beta_H$ is injective. 
\end{proof}

\begin{Corollary}\label{CPdecNor}
Let $(X_\pm)\in\calh_n(X_\pm)$ such that $ \partial_\calh((X_-),(X_+))=0$.  
Suppose that $(X_\pm,\bd X_\pm)$ are \pe pairs of formal dimension $n$ with $[X_\pm]=t((X_\pm))$.
Then $(X,\bd X)$ is a \pe pair of formal dimension~$n$ and $q$ is $\eta$-normal.
Consequently, \eqref{decompEx} is $\eta$-normal \pe decomposition.
\end{Corollary}

\begin{proof}
Since the top line in \eqref{PdecNor-dia} is exact, there is a (possibly non-unique) class  
$(X)\in \calh_n(X,\partial X)$ such that $\beta_\calh((X))=((X_-),(X_+))$. The corollary then follows from
Part (i) of \proref{PPdecNor} (note that the fundamental class $[X]$ is unique since $\beta_H$ is injective).
\end{proof}
\end{ccote}

Part (ii) of \proref{PPdecNor} may be shortened in the following

\begin{Corollary}\label{CPdecNor2}
Suppose that $(X,\bd X)$ and $(X_-,\bd X_-)$ are \pe pairs of formal dimension~$n$ and that 
$q$ and $q_-$ are $\eta$-normal. Then $(X_+,\bd X_+)$ is a \pe pairs of formal dimension~$n$ and 
$q_+$ is $\eta$-normal. In other words, \eqref{decompEx} is an $\eta$-normal \pe decomposition.
\mancqfd
\end{Corollary}

\begin{ccote}\label{coThick} Stable $\eta$-thickenings. \ \rm
Let $Y$ be a connected finite CW-complex of dimension $k-1\geq 2$ and let $\eta$ be a stable vector bundle over $Y$.
An $(Y,\eta)$-referred manifold $(N,g)$ is called an \dfn{$\eta$-thickening of $Y$ of dimension $n$} if 
\begin{itemize}
\item $N$ is a compact connected manifold of dimension $n$ and $g$ is a simple homotopy equivalence;
\item the pair $(N,\bd N)$ is ($n-k$)-connected (with $n-k\geq 2$). 
\end{itemize}
(Recall that $\dim Y = k-1$.) Two $\eta$-thickenings  $(N_1,g_1)$ and $(N_2,g_2)$ of $Y$ are regarded as 
equivalent if there is a degree-one diffeomorphism $h\: N_1\fl{\approx} N_2$ such that $g_2\pcirc h$ is homotopic of $g_1$. 
In the {\it stable range} $n>2(k-1)$, one has the following 

\begin{Proposition}[Wall]\label{PWallThick}
Let $Y$ be a connected finite cell complex of dimension $k-1$ and
let $\eta$ be a stable vector bundle over $Y$. Then, if $n>2(k-1)\geq 6$, there is an unique equivalence class
$\Xi_n(Y,\eta)$ of $\eta$-thickenings of $Y$ in dimension~$n$. \mancqfd
\end{Proposition}

A representative $(N,g)$ of $\Xi_n(Y,\eta)$ is called a (or sometimes {\it the}) \dfn{stable $\eta$-thickening 
of $Y$ of dimension $n$}.

\begin{proof}
This follows from Wall's classification of stable thickenings by their stable tangent bundle
\cite[Proposition~5.1]{WallThick}, whence by their stable normal bundle. More precisely for the uniqueness, 
Wall's theorem provides a diffeomomorphism between any two stable thickenings, which may be not of degree one.
But a stable thickeninbg $N$ is of the form $N_-\times I$ (see \remref{RBdThick} below) and thus $N$ admits a 
self-diffeomorphism of degree $-1$. 
\end{proof}

\begin{Remark}\label{RBdThick}\rm 
When $n>2k-1\geq 6$, the boundary $\bd N$ of a stable $\eta$-thickening $(N,g)$ is a \as{k} manifold 
(See Definition~B in the introduction). 
Indeed, let $\gamma_0\:Y\to N$ be a homotopy inverse of $g$. As $(N,\bd N)$ is ($n-k$)-connected and 
$n-k>k-1$, the map $\gamma_0$ is homotopic to a map $\gamma\:Y\to \bd N$. 
Up to homotopy equivalence, one may assume that $Y$ is a finite simplicial complex and that $\gamma$ is an embedding
into a $\calc^\infty$-triangulation of $\bd N$. Then, $\gamma(Y)$ has a smooth regular neighborhood $N_-$ \cite{HirschNhb}
and $(N_-,g_{|N_-})$ is a stable $\eta$-thickening of $Y$ of dimension $n-1$.  
By the s-cobordism theorem, the manifold $N$ is diffeomorphic to $N_-\times [0,1]$ and one thus gets
a manifold triad $(N,N_-\cup \bd N_-\times [0,1],N_-\times\{1\})$.
But the pair $(N_-,\bd N_-)$ is $(n-k-1)$-connected. 
By the procedure of eliminating handles (see \cite{KeBMS}), $N_-\times \{0\}$ and $N_-\times\{1\}$ both admit a handle decomposition with handles of index $\leq k-1$. Therefore, $\bd N$ is \as{k}. 
\mancqfd
\end{Remark}

The uniqueness of the Spivak class admits a different form for stable thickenings. 

\begin{Lemma}\label{LSpiThick}
Let $(N_j,g_j)$ ($j=1,2$) be two stable $\eta$-thickenings of $Y$ of dimension $n>2(k-1)\geq 6$, 
with Spivak class $(N_j)\in \calh_n(N_j,\bd N_j)$.
Then there exists a simple homotopy equivalence $f\:(N_1,\bd N_1)\to (N_2,\bd N_2)$ satisfying $g_1\simeq f\pcirc g_2$
and $\calh_*f((N_1))=(N_2)$.
\end{Lemma}

\begin{proof}
The data $(N_j,g_j,(N_j))$ determine surgery problems $\alpha_j\:(\caln_j,\bd\caln_j)\to (N_j,\bd N_j)$ such that
$\calh_*\alpha_j((\caln_j)_{TP})=(N_j)$, where $(\caln_j)_{TP}\in\calh_n(\caln_j,\bd\caln_j)$ is 
the Thom-Pontryagin class of $\caln_j$. By Wall's $(\pi-\pi)$-theorem, one may suppose that $f_j$ are 
simple homotopy equivalences. Hence, $(\caln_j,g_j\pcirc f_j)$ are stable $\eta$-thickenings of $Y$. 
By \proref{PWallThick}, there is a diffeomorphism $h\:\caln_1\to\caln_2$ such that $g_2\pcirc\alpha_2\pcirc h$ is homotopic to $g_1\pcirc\alpha_1$. Being a diffeomorphism of degree one, $h$ satisfies $\calh_*h((\caln_1)_{TP})=(\caln_2)_{TP}$.
Therefore the map $f=\alpha_2\pcirc h\pcirc\alpha_1'$, where $\alpha_1'$ is a homotopy inverse of $\alpha_1$, satisfies the conclusion of \lemref{LSpiThick}. 
\end{proof}
\end{ccote}

\begin{ccote}\label{endsMan} Ends of manifolds. \ \rm
A general reference for ends of spaces is \cite{RanEnds}.
We restrict to the case of a $\sigma$-compact finitely dominated open manifolds $U$ with one end $\epsilon$. 
Such an end is called \dfn{tame} \cite[Definition~10, p.~xiii]{RanEnds} if it admits a
sequence $U\supset Z_1 \supset Z_2 \supset \cdots $ of finitely dominated closed neighborhoods with
$$
\bigcap_{j=1}^\infty Z_j = \emptyset \ , \ \pi_1(Z_1)\approx \pi_1(Z_2) \approx  \cdots \approx \pi_1(\epsilon)  \, .
$$

Recall that, according to our definition in~\ref{Pesp}, a \pe pair is finitely dominated.
\begin{Proposition}\label{Pcompletion}
Let $U$ be a connected $\sigma$-compact finitely dominated $n$-dimensional manifold with compact boundary $\bd U$ and with one end $\epsilon$, which is tame and satisfies $\pi_1(\epsilon)\approx \pi_1(U)$. 
Set $\calp=\wa(U)\in \tilde K_0(\pi_1(U),\omega_U)$.
Then there exists a \pe cobordism $(\bbu,\bd U,\bbp)$, 
with $\pi_1(\bbp)\approx\pi_1(\bbu)$, together with a homotopy equivalence $\alpha\:U\to\bbu$.
Moreover, with the identification
$\pi_1(\bbp)\approx\pi_1(\bbu)\approx\pi_1(U)$ obtained using $\alpha$, the formulae
$$
\wa(\bbu)=\calp \ \hbox{ and } \ \wa(\bbp) = \calp + (-1)^{n}\calp^*
$$
hold true in $\tilde K_0(\pi_1(U),\omega_U)$.
\end{Proposition}

\begin{proof}
We use the \dfn{space of ends} $e(U)$ of $U$, which is the space of proper maps $c\:[0,\infty)\to U$ \cite[Definition~1.2]{RanEnds}. 
It is equipped with the \dfn{origin map} $\vartheta\:e(U)\to U$, given by $\vartheta(c)=c(0)$.
Define $\bbu$ to be the mapping cylinder of $\vartheta$:
$$ 
\bbu = e(U)\times [0,1] \dcup U  \big/ \{(c,1)\sim c(0)\} \, ,
$$ 
which retracts by deformation onto $U$, so the inclusion $\alpha\:U\to\bbu$ is a homotopy equivalence.
Set $\bbp=e(U)$. 

The conclusions of the proposition come from \cite[Proposition~10.5]{RanEnds}. We just need to check that  the hypotheses of the latter, i.e. forward and reverse tameness of $\epsilon$, are implied by our tameness hypothesis. This is guaranteed by \cite[Proposition~8.9 and~10.13]{RanEnds}.
\end{proof}
\end{ccote}

\subsection{Cohomology anti-simple manifolds}\label{Scasm}

Let $M$ be a compact connected (smooth or PL)-manifold of dimension $n$ and let $k$ be an integer
with $n\geq 2k$. We shall compare the two Definitions~B of the introduction, i.e. 
$M$ is \dfn{\as{k}} if it admits a handle decomposition without handles of index $j$ 
for $k\leq j \leq n-k$ and $M$ is \dfn{\has{k}} if it is \ccs{k}{n-k}.

Obviously, a \as{k}-manifold is \has{k}. Only the empty manifold is \has{0}. A closed manifold is \as{1} \ 
if and only if it is a homotopy sphere. 

The above definition may also be used for \pe spaces (with or without boundary) of formal dimension $n$. It will be mainly used for closed manifold. In this case, ``cohomology antisimple'' could be called ``homology antisimple'' thanks to the following
 
\begin{Proposition}\label{has/chas}
Let $M$ be a closed connected manifold of dimension $n$ (or a \pe space of formal dimension $n$ without boundary). 
The following condition are equivalent.
\begin{itemize}
 \item[(i)] $M$ is \has{k}.
 \item[(ii)] $H_j(M;R)=0$ for any $\bbz\pi_1(M)$-module $R$ when $k\leq j \leq n-k$.
\end{itemize}
\end{Proposition}

\begin{proof}
Since, by \pe duality, $H^j(M;R) \approx H_{n-j}(M;R)$.
\end{proof}

By \secref{SCSC}, a \has{k} compact manifold $M$ determines a cell-dispensability obstruction $w_k(M)\in\tilde K_0(\pi_1(M))$. 
We call it the \dfn{antisimple obstruction} for $M$, thanks to the following

\begin{Proposition}\label{Pcasas}
Let $M$ be a closed connected manifold of dimension $n\geq 2k \geq 6$ which is \has{k}. 
The following condition are equivalent.
\begin{itemize}
 \item[(a)] $M$ is \as{k}.
 \item[(b)] $w_k(M)=0$.
\end{itemize}
\end{Proposition}

\begin{proof}

If is $M$ is \as{k}, then $M$ has the homotopy type of a finite CW-complex $X$ having no cells in dimension 
$j$ for $k\leq j \leq n-k$. By \lemref{WL2bb}, one has $0=w_k(X)=w_k(M)$, which proves $(a)\Rightarrow (b)$.

Conversely, suppose that $w_k(M)=0$. As $M$ is \has{k} and $k\geq 3$, there is by \proref{WP1}
a homotopy equivalence $f\:Y\to M$, where 
$Y$ is a finite CW-complex without $j$-cells for $k\leq j\leq n-k$. Let $K$ be a finite simplicial complex
homotopy equivalent to $Y^{k-1}$. As $n\geq 2k$ one may suppose that $f$ is an PL-embedding for a $C^\infty$-triangulation
of $M$. The complex $f(K)$ admits a smooth regular neighborhood $T$, which admits a handle decomposition 
with handles of index $\leq k-1$. Let $V=M-{\rm int\,}T$, which we see as a cobordism between $N={\rm Bd}T$
and the empty set.
As $n\geq 2k \geq 6$, one has $\pi_1(N)\approx\pi_1(V)$ and 
$H_j(\tilde V,\tilde N)\approx H_j(\tilde M,\tilde T)=0$ for $j\leq n-k$. The procedure of eliminating handles 
in a cobordism (see \cite{KeBMS}) then produces a handle decomposition of $(V,N)$ with only handles of index $> n-k$,
thus proving (a).
\end{proof}

We now turn our attention to realizing elements of $\tilde K_0(\bbz\pi)$ as antisimple obstructions $w_k(M)$ for a
\has{k} manifold $M$ with $\pi_1(M)=\pi$. We first establish necessary conditions to be fulfilled by $w_k(M)$. 
If $P$ is a finitely generated projective $\bbz\pi$-module, so is its \dfn{dual} $P^*=\hom_{\bbz\pi}(P,\bbz\pi)$,
endowed with the left $\bbz\pi$-action  $(a\beta)(u)=\beta(u)\bar a$ (recall that the involution $a\mapsto\bar a$
on $\bbz\pi$ involves the orientation character $\omega$ of $M$: see \ref{holoc}).
This induces an involution $\calp\mapsto\calp^*$ on $\tilde K_0(\bbz\pi)$. An element $\calp\in\tilde K_0(\bbz\pi)$
is called \dfn{$n$-self-dual}, i.e $\calp=(-1)^{n+1}\calp^*$

\begin{Proposition}\label{Pdua}
Let $M$ be a closed connected manifold of dimension $n\geq 6$ which is \has{k} for $k\geq 3$. 
Then $w_k(M)$ is $n$-self-dual.
\end{Proposition}

\begin{proof}
Consider a handle decomposition $\calh$ for $M$
$$
\calh : \ D^n = \calh_0 \subset \calh_1 \subset \cdots \subset \calh_n = M \ ,
$$
where $\calh_j$ is the union of handles of index $\leq j$ in $\calh$. This handle decomposition 
makes $M$ homotopy equivalent to a cell complex and thus gives rise to 
a chain complex of free $\bbz\pi$-modules $C_j(\calh)=H_j(\tilde \calh_j,\tilde \calh_{j-1})$ 
whose homology is that of $\tilde M$. As seen in \lemref{LW1b} and its proof, one has
\beq{Pdua-e1}
w_k(M) = (-1)^k[B_{k-1}(\calh)]\in \tilde K_0(\bbz\pi) \, .
\eeq
We now use the ``dual'' handle decomposition $\calh^*$ of $\calh$ (see e.g. \cite[p.~394]{MilnorWT}), producing a 
chain complex $C_j(\calh^*)$, whose homology is also that of $\tilde M$. By \remref{Rj>k}, one has 
\beq{Pdua-e2}
w_k(M) = w_{n-k}(M) = (-1)^{n-k}[B_{n-k-1}(\calh^*)] \in \tilde K_0(\bbz\pi) \, .
\eeq
The correspondence $\Theta$ associating to a $j$-handle $e$ its dual ($n-j$)-handle $\Theta(e)$ 
produces a chain isomorphism
$$
\dia{
C_{n-k+1}(\calh^*) \ar[r]^{\partial} \ar[d]^{\Theta}_\approx & C_{n-k}(\calh^*) \ar[d]^{\Theta}_\approx 
\ar[r]^{\partial} &  C_{n-k-1}(\calh^*) \ar[d]^{\Theta}_\approx 
\\
C_{k-1}(\calh)^* \ar[r]^{\partial^*}   & C_{k}(\calh)^* \ar[r]^{\partial^*} 
&  C_{k+1}(\calh)^* 
}
$$
Therefore,
\beq{Pdua-e3}
B_{n-k-1}(\calh^*) \approx {\rm Image\,}\big(\partial^*\: C_{k}(\calh)^* \to C_{k+1}(\calh)^*\big) \, .
\eeq
Since $B_{k-1}(\calh)$ is projective, the exact sequence
\beq{Pdua-e4}
0 \to Z_k(\calh) \to C_k(\calh) \fl{\partial} B_{k-1}(\calh) \to 0
\eeq
splits. Passing to the dual modules, we get the split exact sequence
\beq{Pdua-e5}
0 \to B_{k-1}(\calh)^* \fl{\partial^*} C_k(\calh)^* \to Z_k(\calh)^* \to 0  \ . 
\eeq
Therefore,
$$
\begin{array}{rcll}
 (-1)^{n-k} w_k(M) &=& [B_{n-k-1}(\calh^*)] & \comeq{by \eqref{Pdua-e2}} \\[2mm] &=&
 [Z_k(\calh)^*] & \comeq{by \eqref{Pdua-e3} and \eqref{Pdua-e5} } \\[2mm] &=&
 [Z_k(\calh)]^* \\[2mm] &=&
 -[B_{k-1}(\calh)]^*  & \comeq{by \eqref{Pdua-e4}} \\[2mm] &=&
 (-1)^{k+1} w_k(M)^*  & \comeq{by \eqref{Pdua-e1}  ,}
\end{array}
$$
proving that $w_k(M)=(-1)^{n+1}w_k(M)^*$. 
\end{proof}

\subsection{Realization of antisimple obstructions}\label{Srea}

Let $\pi$ be a finitely presented group and let $\omega\:\pi\to\{\pm 1\}$ be a homomorphism.
In this section, we shall realize classes in $\tilde K_0(\bbz\pi,\omega)$ as antisimple obstruction of a 
\has{k} closed manifold $M$. Making sense of this question requires some identification of $(\pi_1(M),\omega_M)$
with $(\pi,\omega)$. More strongly, we will fix the stable normal ($k-1$)-type of $M$, using the notion
of an $(Y,\eta)$-referred manifold (see~\ref{MPfnormal}). We will thus consider the following

\begin{Problem}\label{ProbRea}
Let  $Y$ be a connected finite cell complex of dimension $k-1$ and let $\eta$ be a stable vector bundle over $Y$.
Let $\calp\in\tilde K_0(\bbz\pi_1(Y),\omega_\eta)$. Does there exist an
$(Y,\eta)$-referred closed manifold $(M,g)$ such that $M$ is \has{k} and satisfies $g_*(w_k(M))=\calp$?
\end{Problem}

Problem~\ref{ProbRea} will partly answered in \thref{Trea} below (see also \proref{PreaPoin}).
Before going to this, we show that a $(Y,\eta$)-referred compact manifold enjoys a decomposition involving
a stable $\eta$-thickening of $Y$ in the sense of~\ref{coThick}.

\begin{Proposition}\label{Pdecomp}
Let $Y$ be a connected finite cell complex of dimension $k-1$ and
let $\eta$ be a stable vector bundle over $Y$.
Let $(M,g)$ be an $(Y,\eta$)-referred compact manifold of dimension $r\geq 2k\geq 6$. 
Then, there are codimension-$0$ compact submanifolds $N$ and $T$ of $M$ giving a decomposition
\beq{Pdecomp-dec}
M= N\cup T \ , \ N\cap T = \bd N \ , \ \bd T = \bd  M \dcup \bd N  \ ,
\eeq
such that $(N,g_{|N})$ is a stable $\eta$-thickening of $Y$ and $g_{g|T}$ is an $(Y,\eta$)-referred map.
Moreover, for any $\bbz\pi_1(Y)$-module $R$, the restriction homomorphism \\
$\rho^j\:H^j(M,T;R)\to H^j(M;R)$
satisfies the following properties:
\begin{itemize}
 \item[(i)] $\rho^j$ is an isomorphism for $k\leq j\leq r-k-1$ and $\rho^{r-k}$ is injective;
 \item[(ii)] if $M$ is a closed manifold, then $\rho^j$ is an isomorphism for $k\leq j\leq r-k$.
 Consequently, $M$ is \has{k} if and only if
$T$ is \has{k}, in which case $g_*(w_k(M))=(g_{|T})_*(w_k(T))$ in $\tilde K_0(\bbz\pi_1(Y),\omega_\eta)$.
\end{itemize}
\end{Proposition}

\begin{proof}
The $(Y,\eta)$-reference map $g\:M\to Y$ admits a homotopy section $\gamma\:Y\to M$ 
(see \lemref{Lrefretr} and its proof).
Up to homotopy equivalence, one may assume that $Y$ is a finite simplicial complex and that $\gamma$ is an embedding
into a $\calc^\infty$-triangulation of $M$ with $\gamma(Y)\cap \partial M=\emptyset$. Then, $\gamma(Y)$ has a smooth regular neighborhood $N$ in $M\setminus \partial M$ \cite{HirschNhb}.
Hence, the embedding $\gamma$ factors through an embedding $\gamma^1\:Y\to N$ which is a simple homotopy equivalence,
and $g_{|N}\:N\to Y$ is a homotopy inverse of $\gamma^1$. It follows that $g_{|N}$ is a homotopy equivalence,
which is simple by the composition rule for the Whitehead torsion \cite[Lemma~7.8]{MilnorWT}.
Since $g$ is $\eta$-normal, so is $g_{|N}$.
Since $r\geq 2k\geq 6$, one has $\pi_1(\bd N)\approx\pi_1(N)$ and, by \pe duality,
$H_j(\tilde N,\widetilde{\bd N})\approx H^{r-j}(N;\bbz\pi_1(N))\approx H^{r-j}(Y;\bbz\pi_1(Y))=0$ for $j\leq n-k$.
Therefore, the pair $(N,\bd N)$ is ($r-k$)-connected and $(N,g_{|N})$ is a stable $\eta$-thickening of $Y$. 
We thus get the decomposition \eqref{Pdecomp-dec} by setting $T=M\psetminus {\rm int}N$.

By van Kampen theorem, $\pi_1g_{g|T}\:\pi_1(T)\to\pi_1(Y)$ is an isomorphism. Since
$H_j(\tilde M,\tilde T)\approx H_j(\tilde N,\widetilde{\bd N})$ the pair $(M,T)$ is ($r-k$)-connected.
Therefore, $g_{|T}$ is an $Y$ reference map, which is $\eta$-normal since $g$ is so.

By excision and \pe duality, one has isomorphisms
\beq{Pdecomp-eq0}
H^j(M,T;R) \fl{\approx} H^j(N,\bd N;R) \fl{\approx} H_{r-j}(N;R) = 0 \ \hbox{ for } \ j\leq r-k \, ,
\eeq
which proves (i). 

When $\bd M=\emptyset$, we claim that the restriction homomorphism $\beta\:H^*(M,T;R)\to H^*(M;R)$ is injective. 
Indeed, \eqref{Pdecomp-eq0} sits in a commutative diagram
$$
\dia{
H^{j}(M,T;R) \ar[d]^\beta \ar[r]^(0.46){\rm\tiny excis.}_(0.45)\approx &  H^{j}(N,\bd N;R) 
\ar[r]^(0.54){\rm PD}_(0.54)\approx & H_{r-j}(N;R)  \ar[d]^{\sigma_*}
\\  
H^{j}(M;R) \ar[rr]^{\rm PD}_\approx && H_{r-j}(M;R) 
}
$$
where $\sigma_*$ is induced by the inclusion $\sigma\:N\to M$. As the reference map $g\:M\to Y$ produces
a homotopy retraction of $\sigma$, the homomorphism  $\sigma_*$ is injective. Therefore,
the cohomology sequence of the pair $(M,T)$ splits into short exact sequences
$$
0 \to H^j(M,T;R) \to H^j(M;R) \to H^j(T;R) \to 0 \, .
$$
This, together with \eqref{Pdecomp-eq0} proves the first assertion of (ii). 
Finally, one has $H_*(\tilde M,\tilde T)\approx H_*(\tilde N,\widetilde{\bd N})$, so the pair $(M,T)$ is ($r-k$)-connected.
As $r-k\geq k$, one has $g_*(w_k(M))=(g_{|T})_*(w_k(T))$ by \lemref{WL2bb}.
\end{proof}

Our first result concerning Problem~\ref{ProbRea} is an affirmative answer in the category of \pe spaces. Recall that, according to our definition in~\ref{Pesp}, a \pe space or pair is finitely dominated.

\begin{Proposition}\label{PreaPoin}
Let $Y$ be a connected finite cell complex of dimension $k-1\geq 2$ and let $\eta$ be a stable vector bundle over $Y$.
Let $\calp\in \tilde K_0(\bbz\pi_(Y),\omega_\eta)$ and let $r\geq 2k$ be an integer.
Then, there exists an $(Y,\eta$)-referred \pe space 
$(\bbp,q)$ of formal dimension $r$, 
such that $\bbp$ is \has{k} satisfying $w_k(\bbp)=\calp$ and $\wa(\bbp) = \calp + (-1)^{r}\calp^*$ 
(see Convention~\ref{Conv} below). 
Moreover, $\bbp$ admits up to homotopy an $\eta$-normal \pe decomposition (see~\ref{PdecNor})
\beq{PreaPoin-dec}
\bbp \simeq N \cup \bbt \ , \ N\cap \bbt = \bd N  = \bd\bbt \ ,
\eeq
where $(N,q_{|N})$ is a stable $\eta$-thickening of dimension $r$ of $Y$.
The \pe space $\bbt$ 
is \has{k} and satisfies $w_k(\bbt)=w_k(\bbp)$ and $\wa(\bbt) = \wa(\bbp)$. 
\end{Proposition}

\begin{Convention}\label{Conv}\rm
Let $Y$ be a connected finite cell complex. 
Let $(A,\alpha)$ be an $(Y,\eta)$-referred space. Then, $\tilde K_0$-equalities are understood to hold in 
$\tilde K_0(\pi_1(Y),\omega_\eta)$, using the map $\alpha$ and its restrictions to subspaces of $A$. For instance, 
the equation $w_k(\bbt)=w_k(\bbp)$ in \proref{PreaPoin} should be understood as $(q_{|\bbt})_*(w_k(\bbt))=q_*(w_k(\bbp))$.
\mancqfd
\end{Convention}

\begin{proof}
We set $(\pi,\omega)=(\pi_1(Y),\omega_\eta)$ and we use Convention~\ref{Conv} throughout the proof without notice.
Consider the infinite $k$-dimensional CW-complex $L$ of \proref{PW4}. By \remref{WR10}, 
$L$ is the direct limit $L_\infty$ of a direct system 
\beq{Trea-e0}
Y\hookrightarrow L_1\hookrightarrow L_2 \hookrightarrow \cdots \hookrightarrow L_\infty=L \fl{g} Y
\eeq
of finite cell complexes of dimension $k$. The $Y$-referred map $g\:L\to Y$ is the direct limit of 
$Y$-referred maps $g_j\:L_j\to Y$. The composed map in \eqref{Trea-e0} is ${\rm id}_Y$. 

We now see that Direct system \eqref{Trea-e0} may be realized up to homotopy equivalence by a system of 
($r+1$)-dimensional manifolds (compact with boundary when $j<\infty$), getting a commutative diagram
\beq{Trea-e1}
\dia{ 
Y \ar@{^{(}->}[r] \ar@{<-}[d]^(.5){\phi_0}
& 
L_1 \ar@{^{(}->}[r] \ar@{<-}[d]^(.5){\simeq}_(.5){\phi_1} & L_2 \ar@{^{(}->}[r]
\ar@{<-}[d]^(.5){\simeq}_(.5){\phi_2}
& \cdots \ar@{^{(}->}[r] & L_\infty=L     \ar@{<-}[]!<3.3ex,-2ex>;[d]!<3.3ex,1ex>^(.5){\simeq}_(.5){\phi} 
\ar[r]^(0.60)g & Y
\\ 
U_0 \ar@{^{(}->}[r] & U_1 \ar@{^{(}->}[r]^{\alpha_0} & U_2 \ar@{^{(}->}[r]^{\alpha_1} & \cdots \ar@{^{(}->}[r] & U_\infty=U
}
\eeq
where the vertical arrows are homotopy equivalences (simple for $j<\infty$).
We start with $(U_0,\phi_0)$ being a stable $(Y,g_0^*\eta)$-thickening of dimension $r+1$. We can thus choose a stable
isomorphism from the stable normal bundle on $U_0$ and $\phi_0^*(g_0^*\eta)$, or, equivalently, 
a stable trivialization $\calf_0$ of $TU_0\oplus \phi_0^*(g^*\eta)$ 
(i.e. $(\phi_0,\calf_0)$ is a {\it normal map} in the language of \cite{WallBookSurgery}). 
Recall that the inclusions in \eqref{Trea-e0} are obtained by attaching cells of dimension $k-1$ and $k$.
By surgery below the middle dimension
on $\bd U_0$, \cite[Chapter~1]{WallBookSurgery}, the data $(U_0,\phi_0,\calf_0)$ thus determines $(U_1,\phi_1,\calf_1)$, then
$(U_2,\phi_2,\calf_2)$, etc. Note that $(U_j,\phi_j)$ is the stable 
$(L_j,g_j^*\eta)$-thickening of dimension $r+1$ and $\calf_j$ (obtained from $\calf_0$ without further choice) 
determines the embedding $\alpha_j\:U_j\to U_{j+1}$.

At the limit when $j\to\infty$, one gets an open manifold $U$ together with an 
$(Y,\eta)$-reference map  $g\pcirc\phi\: U\to Y$. 
The manifold $U$ has one end $\epsilon$, determined by the cofinal system of neighborhoods 
$E_j=U\psetminus {\rm int} U_j$. The inclusions $E_{j+1}\hookrightarrow E_j\hookrightarrow U$ induce an isomorphism
on the fundamental groups. As $U\simeq L$ is finitely dominated by \proref{PW4} and $\bd E_j$ is a closed manifold,
the open manifold $E_j$ is finitely dominated (see \cite[Complement~6.6]{SieTh}). Therefore, 
the end $\epsilon$ is {\it tame} (see~\ref{endsMan}).

As mentioned in \remref{RWallObs}, $w_k(L)=w_k(U)$ is Wall's finiteness obstruction $\wa(U)$. 
One has $\wa(U)=\wa(Z_i)$ by the sum theorem \cite[Theorem~6.5]{SieTh} and thus, by \cite[Proposition~6.11]{SieTh},
$\wa(U)$ coincides with $\sigma(\epsilon)$, the Siebenmann obstruction  to complete $U$ into a  
compact manifold with boundary. 

By \proref{Pcompletion}, there exists a \pe pair $(\bbu,\bbp)$, 
with $\pi_1(\bbp)\approx\pi_1(\bbu)$, together with a homotopy equivalence $\alpha\:U\to\bbu$.
Recall from \proref{PW4} that $\bbu\simeq L$ is is \cs in degrees $\geq k$. By  \pe duality, one has
$$
H_j(\widetilde{\bbu},\tilde\bbp) \approx H_j(\bbu,\bbp;\bbz\pi) \approx H^{r+1-j}(\bbu;\bbz\pi) = 0  \quad {\rm for } \
j\leq n+1-k \, .
$$
Hence, the pair $(\bbu,\bbp)$ is ($r+1-k$)-connected. Since $U\simeq \bbu$, the map $g\pcirc\phi\: U\to Y$ extends 
to an $Y$-reference map $q\:\bbu\to Y$. By the connectivity of $(\bbu,\bbp)$, the restriction of $q$ to $\bbp$ 
(also called $q$) is an $Y$-reference map. Also, this connectivity implies that $\bbp$ is \has{k} and that 
$w_k(\bbp)=w_k(L)=\calp$.
As $\alpha\:U\to \bbu$ is a homotopy equivalence, one has $\wa(\bbu)=\wa(U)=\wa(L)=\calp$. The formula 
\beq{wobd}
\wa(\bbp) = \calp + (-1)^{r}\calp^*
\eeq
is asserted in \proref{Pcompletion}.

We now prove that the reference map $q\:\bbu\to Y$ is $\eta$-normal (and so is $q\:\bbp\to Y$ by Part (3) of 
\lemref{Lnormal}). We use the ideas of Pedersen-Ranicki \cite[proof of Lemma~3.1]{PeRa} (with different notations).

By Siebenmann's famous theorem \cite[Theorem~7.5]{SieTh}, 
the open manifold $U\times S^1$ is the interior of a compact manifold $A$ with boundary
$V=\bd A$. The \pe duality argument used above to prove that the pair $(\bbu,\bbp)$ is ($r+1-k$)-connected proves that 
$(A,V)$ is ($r+1-k$)-connected.
The projections of $U\times S^1$ onto $S^1$ together with the map $g\pcirc\phi\:U\to Y$ gives rise to a map 
$g_A\:A\to Y$ which is $\eta$-normal and a map $p_A\:A\to S^1$. Let $(\hat A,\hat V)\to (A,V)$ be the infinite cyclic cover associated to the map $p_A$.
It is proven in \cite[proof of Lemma~3.1]{PeRa} that there are homotopy equivalences 
\begin{itemize}
 \item[(1)] $h_1\: (A,V)\fl{\simeq} (\bbu,\bbp)\times S^1$
 \item[(2)] $h_2\: \hat{V} \fl{\simeq} \bbp$.
\end{itemize}
The composite map
\beq{Prea-eq7}
A \fl{h_1} \bbu\times S^1 \fl{h_1'} A \fl{g_A}  Y \, , 
\eeq
where $h_1'$ is a homotopy inverse of $h_1$, is homotopic to $g_A$ and is thus $\eta$-normal. By \lemref{Lnormal},
we deduce that $g_A\pcirc h_1'$ is $\eta$-normal, and so is its restriction to $\bbu\times pt$. 
But this restriction is homotopic to the $Y$-reference map $q\:\bbu\to Y$, which is thus $\eta$-normal.

It remains to construct the decomposition of \eqref{PreaPoin-dec}. We check that the map 
$\bar g_A= (g_A,p_A)\: A\to Y\times S^1$ is an $(Y\times S^1,\bar\eta)$-reference map for $A$,
where $\bar\eta$ is induced from $\eta$ by the projection to $Y$. As the pair $(A,V)$ is 
is ($r+1-k$)-connected, the restriction of $\bar g_A$ to $V$ is a an $(Y\times S^1,\bar\eta)$-reference map for $V$.
By \proref{Pdecomp}, there is a manifold decomposition $V= N_0 \cup T_0$, where $(N_0,\bar g_A)$ is a stable 
$\bar\eta$-thickening of $Y\times S^1$. As $\bar g_A|U\times S^1 \simeq q\times {\rm id}_{S^1}$,
\proref{PWallThick} implies that $N_0$ is diffeomorphic to $N\times S^1$, 
where $(N,q)$ is a stable $\eta$-thickening of $Y$.

Using the homotopy equivalence $h_2\:\hat V \to \bbp$, one thus gets a decomposition 
$\bbp\simeq (N\times \bbr) \cup \hat T_0$.
The manifold $N$ is diffeomorphic to $N\cup_{\bd N} C$ where $C=\bd N \times [0,1]$. Form the space
$$
\bbt =  \big(C\times \{0\}\big) \cup \hat T_0  \subset (N\times\bbr) \cup \hat T_0   \, .
$$
containing a subspace $\bd\bbt=\bd (N\times \{0\})$. One has the decomposition 
$$
\bbp \simeq (N\times\{0\})\cup \bbt\simeq N\cup\bbt \, .
$$

As the inclusions $\bd N\hookrightarrow N\hookrightarrow \bbp$ induce isomorphisms 
$\pi_1(\bd N)\approx\pi_1(N)\approx\pi_1(\bbp)$, it follows from van Kampen's theorem and the property of
amalgamated products (see \cite[Theorem~2.6, p.~187]{LySchu}) that $\pi_1(\bd N)\approx\pi_1(\bbt)\approx\pi_1(\bbp)$.
By excision, $H_*(\tilde N,\widetilde{\bd N})\approx H_*(\tilde \bbp,\tilde \bbt)$. Therefore, the pair $(\bbp,\bbt)$ is ($r-k$)-connected and thus the map of $q_{|\bbt}\:\bbt\to Y$ is an $Y$-reference map.

As the map $q\:N\to Y$ is $\eta$-normal, that \eqref{PreaPoin-dec} is an $\eta$-normal \pe decomposition
follows from \corref{CPdecNor2}.

Finally, the equation $\wa(\bbt) = \wa(\bbp)$
comes from the sum theorem \cite[Theorem~6.5]{SieTh}. That $\bbt$ is \has{k} and the 
equation $w_k(\bbt)=w_k(\bbp)$ are proven by the corresponding arguments given in the proof of 
Part (ii) of \proref{Pdecomp}.
\end{proof}

We now turn our attention to Problem~\ref{ProbRea} as it is stated, that is in the category of closed manifolds.
Let $Y$ be a connected finite cell complex of dimension $k-1$ and
let $\eta$ be a stable vector bundle over $Y$. Set $(\pi,\omega)=(\pi_1(Y),\omega_\eta)$.
Let $(M,g)$ be an $(Y,\eta$)-referred closed manifold of dimension $r$ such that $M$ is \has{k}.
By \proref{Pdua}, its antisimple obstruction is $r$-self-dual, i.e. satisfies $g_*(w_k(M))=(-1)^{r+1}g_*(w_k(M))^*$.
It thus defines a class
\beq{tate}
[g_*(w_k(M))] \in H^{r+1}(\bbz_2;\tilde K_0(\bbz\pi,\omega)) 
\eeq
in the Tate cohomology group 
\small
$$
H^{r+1}(\bbz_2;\tilde K_0(\bbz\pi,\omega)) = 
\{\calp\in \tilde K_0(\bbz\pi,\omega)\mid \calp=(-1)^{r+1}\calp^*\}\big/ \{\calp+(-1)^{r+1}\calp^*\} \, .
$$
\rm
This Tate cohomology group is an abelian group of exponent two, which occurs 
in the Ranicki exact sequence
\beq{ranicki}
L_{r+1}^h(\pi,\omega) \to L_{r+1}^p(\pi,\omega) \to
H^{r+1}(\bbz_2;\tilde K_0(\bbz\pi,\omega)) \fl{\delta_{R}} L_{r}^h(\pi,\omega) \to L_{r}^p(\pi,\omega)
\eeq
where $L_j^h$ (respectively: $L_j^p$) are the surgery obstruction groups for surgery data with target
a finite (respectively: finitely dominated) \pe complex. This sequence was first obtained by Ranicki
using his algebraic setting of surgery \cite[Theorem~4.3]{RanFound}. A geometrical version  
of Sequence~\eqref{ranicki} will be used, which was provided by Pedersen-Ranicki \cite[p.~243]{PeRa}.
Our partial solution of Problem~\ref{ProbRea} is the following

\begin{Theorem}\label{Trea}
Let $Y$ be a connected finite cell complex of dimension $k-1\geq 2$ and let $\eta$ be a stable vector bundle over $Y$.
Set $(\pi,\omega)=(\pi_1(Y),\omega_\eta)$ and let $r\geq 2k$ be an integer.
Let $\calp\in \tilde K_0(\bbz\pi,\omega)$ be an $r$-self-dual element.
Suppose that 
$[\calp]\in H^{r+1}(\bbz_2;\tilde K_0(\bbz\pi,\omega))$ belongs to the kernel of the homomorphism $\delta_R$
of \eqref{ranicki}.
Then, there exists an $(Y,\eta$)-referred closed manifold $(M,q)$ of dimension $r$ 
such that $M$ is \has{k} and satisfies $q_*(w_k(M))=\calp$.
\end{Theorem}

\begin{proof}
We consider the \pe pair $(\bbu,\bbp)$ of the proof of \proref{PreaPoin}, equipped with its 
$(Y,\eta)$-reference map $q\:\bbu\to Y$ and its restriction $q\:\bbp\to Y$.
One has $\wa(\bbu)=\calp$ and $\wa(\bbp) = \calp + (-1)^{r}\calp^*$ (using Convention~\ref{Conv}).
Since $\calp=(-1)^{r+1}\calp^*$ by hypothesis, 
one has $\wa(\bbp)=0$ and, by \cite[Theorem~F]{WallFin}, $\bbp$ is homotopy equivalent to a finite complex.
Thus, by changing the pair $(\bbu,\bbp)$ by a homotopy equivalence, one may suppose 
that  $\bbp$ is a {\it finite} \pe space. 

Being $\eta$-normal, the map $q$ determines a surgery problem with target $\bbu$ (see~\ref{MPfnormal}).
In this language and using the point of view of \cite[pp.~242-44]{PeRa}, 
the data $(\bbu,q)$ represents a class
in the geometric L-group $L_{r+1}^{1,p,h}(Y)$. The definition of the latter is akin to that of 
$L_{r+1}^{1}(Y)=L_{r+1}^{1,s}(Y)$ (also working for $L_{r+1}^{1,h}(Y)$) given in \cite[Chapter~9]{WallBookSurgery}.
With these geometric L-groups, Pedersen and Ranicki obtain, for $r\geq 5$, an isomorphism $\sigma_*$ of exact sequences
\small
\beq{pedranicki}
\dia{
L_{r+1}^{1,h}(Y) \ar[r] \ar[d]_\approx^(0.45){\sigma^h}    & 
L_{r+1}^{1,p}(Y) \ar[r] \ar[d]_\approx^(0.45){\sigma^p} &
L_{r+1}^{1,p,h}(Y) \ar[r]^(0.5){\delta_{PR}} \ar[d]_\approx^(0.45){\sigma^{p,h}}&  
L_{r}^{1,h}(Y) \ar[r] \ar[d]_\approx^(0.45){\sigma^h} & 
L_{r}^{1,h}(Y) \ar[d]_\approx^(0.45){\sigma^p}
\\  
L_{r+1}^h(\pi) \ar[r] & L_{r+1}^p(\pi) \ar[r] &
H^{r+1}(\bbz_2;\tilde K_0(\bbz\pi)) \ar[r]^(0.65){\delta_{R}} &  L_{r}^h(\pi) \ar[r] & L_{r}^p(\pi)
}
\eeq  
\rm
where the bottom line is Ranicki's exact sequence~\eqref{ranicki}, writing $L^*_*(\pi)$ for $L^*_*(\pi,\omega)$
and $\tilde K_0(\bbz\pi)$ for $(\tilde K_0(\bbz\pi),\omega)$.
The arrows in \eqref{pedranicki} satisfy the following properties
\begin{itemize}
 \item $\sigma^{p,h}(\bbu,q) = [\wa(\bbu)]$;
 \item $\delta_{PR}$ sends the class of $(\bbu,q)$ to that of $(\bbp,q)$.
  \item $\sigma^h$ is the surgery obstruction measured in $L_r^h(\pi)$.
\end{itemize}
Hence, in the language of~\ref{MPfnormal}, one has 
$$
\sigma(\bbp,q)=\sigma^h\pcirc\delta_{PR}(\bbu,q) = \delta_R\pcirc\sigma^{p,h}(\bbu,q) = \delta_R(\calp) \, . 
$$
By hypothesis, $[\wa(\bbu)]=[\calp]\in\ker\delta_R$, hence, $\sigma(\bbp,q)=0$.
This means that there exists a $q^*\eta$-normal homotopy equivalence $\theta\:M\to\bbp$ 
where $M$ is a closed $r$-manifold. The latter thus admits the $(Y,\eta)$-reference map $q\pcirc\theta$.
Being homotopy equivalent to $\bbp$, the closed manifold $M$ is \has{k} and $w_k(M)=w_k(\bbp)=\calp$ by \lemref{WL2bb}.
\end{proof}

\subsection{Cohomology antisimple cobordisms}\label{SScob}

Let $n\geq 2k\geq 6$ be integers. A (compact) cobordism $(W,M,M')$ is called \dfn{\has{k}} if the following three conditions hold.
\begin{itemize}
\item $M$ and $M'$ are closed connected manifold of dimension $n$ which are \has{k}; 
\item The pairs $(W,M)$ and $(W,M')$ are weakly ($k-1$)-connected (see \ref{connectivity});
\item $W$ is \ccs{k}{n-k}. 
\end{itemize}
For example, an $h$-cobordism between closed connected \has{k} mani\-folds is \has{k}.

Recall from \eqref{tate} that a closed connected manifold $M$ of dimension $n$ which is \has{k} 
defines a class 
$[w_k(M)] \in H^{n+1}(\bbz_2;\tilde K_0(\bbz\pi_1(M),\omega_M))$.

\begin{Proposition}\label{PcobL2}
Let $(W,M,M')$ be a \has{k} cobordism between closed connected \has{k} manifolds of dimension $n\geq 2k\geq 6$. 
Then $[w_k(M)]=[w_k(M')]$ in $H^{n+1}(\bbz_2;\tilde K_0(\bbz\pi_1(W),\omega_W)$. 
\end{Proposition}

\begin{proof}
The proof involves two steps.

\sk{1}\noindent{\it Step 1: Reduction to the case where $(W,M)$ and $(W,M')$ are ($k-1$)-connected.} \  
The cobordism $(W,M,M')$ admits a handle decomposition of the form
$$
M\times I = \calh_{k-1} \subset  \calh_k \subset\cdots\subset \calh_{n+2-k} = W \ ,
$$
where $\calh_j$ is the union of handles of index $\leq j$. Indeed, since the pair $(W,M)$ is weakly ($k-1$)-connected, it is 
($k-2$)-connected and the procedure of eliminating handles in a cobordism (see \cite{KeBMS}), 
applied to a given handle decomposition $\calh$, permits us to 
to get rid of the handles of index $<k-1$ and $>n+2-k$. The same procedure applied to the dual handle decomposition $\calh'$
makes it of the form
$$
M'\times I = \calh_{k-1}' \subset  \calh_k' \subset\cdots\subset \calh_{n+2-k}' = W \ .
$$
The manifold $V=\calh_{k-1}$ and $V'=\calh_{k-1}'$ give cobordisms $(V,M,M_1)$ and $(V',M',M_1')$ having 
handle decomposition involving only handles of index $k-1$. One has $W=V\cup W_1\cup V'$, 
defining a cobordism $(W_1,M_1,M_1')$. The inclusions of the above manifolds into $W$ all induce isomorphisms on
fundamental groups with orientation characters, so all these will be identified with $(\pi,\omega)=(\pi_1(W),\omega_W)$.
One has the following facts.
\begin{itemize}
\item[(1)] {\it $V$ is \ccs{k}{n-k}.} \ Indeed, for any $\bbz\pi$-module $R$, 
the exact sequence $H^{j}(V,M;R) \to H^j(V;R) \to H^j(M;R)$ implies the claim, since
$M$ is \has{k} and $H^{j}(V,M;R)=0$ if $j\neq k-1$.
\item[(2)] {\it $w_k(V)=w_k(M)$.} \ Let $N$ be the union of handles of index $\leq k-1$ in a handle decomposition of $M$
and let $T=M\psetminus {\rm int N}$. Since the pair $(M,N)$ is ($k-1$)-connected, the 
the ($k-1$)-handles of $V$ might be attached on $N$, producing a cobordism $N_+$ from $N$, and 
$V=N_+ \cup_{\bd N \times I} (T\times I)$. The pair $(V,N_+)$ is ($k-1$)-connected and 
$$
H_k(\tilde V,\tilde N_+) \approx H_k(\tilde T,\widetilde{\bd N}) \approx H_k(M,N) \, .  
$$
As $w_k(V)=[H_k(\tilde V,\tilde N_+)]$ and $w_k(M) =[H_k(\tilde M,\tilde N)]$, the claim follows.
\item[(3)] {\it $M_1$ is \has{k} and $w_k(M_1)=w_k(M)$.} \ This follows from (1) and (2), since $(V,M_1)$
is ($n+2-k$)-connected and $n+2-k>n-k>k$. 
\item[(4)] {\it $W_1$ is \ccs{k}{n-k}.} \ As $(W,W_1)$ is ($n+2-k$)-connected,
this follows from the corresponding property of $W$. 
\end{itemize}
Points (1)--(3) also hold true for $(V',M')$. Therefore, it is equivalent to prove \proref{PcobL2} for $(W,M,M')$
or for  $(W_1,M_1,M_1')$, so we may assume that $(W,M)$ and $(W,M')$ are ($k-1$)-connected.

\sk{1}\noindent{\it Step 2: Proof of \proref{PcobL2} when $(W,M)$ and $(W,M')$ are ($k-1$)-connected.} \ 
Applying \lemref{LXT} to both ends of $W$ gives the equalities
\beq{PcobL2-e1}
w_k(M) + [H_k(W,M)] = w_k(W) = w_k(M') + [H_k(W,M')]  \, .
\eeq
We will compute  $[H_k(W,M')]$. As in Step 1, one checks that the 
cobordism $(W,M,M')$ admits a handle decomposition of the form
$$
M\times I = \calh_k \subset  \calh_{k+1} \subset\cdots\subset \calh_{n+1-k} = W \ ,
$$
where $\calh_j$ is the union of handles of index $\leq j$. 
Since $n\geq 2k$, one has $n+2-k>k$. The chain complex $C_*(\tilde W,\tilde M)$ is thus the equivalent to 
$$
0 \to C_{n+1-k}(\tilde \calh) \to\cdots\to C_k(\tilde \calh) \to 0 \, .
$$
Write $C_j=C_j(\tilde\calh)$, $Z_j=Z_j(\tilde\calh)$ and $B_j=B_j(\tilde\calh)$.
For $k<j<n+1-k$, one has $H_j(\tilde W)=0=H_j(\tilde M)$ by \lemref{LW1}.
Therefore, for $k<j<n+1-k$, $H_j(\tilde W,\tilde M)=0$, and thus $Z_j=B_j$. 
One then has exact sequences
$$
\begin{array}{lll}
 0\to B_{k}\to C_k \to H_k(\tilde W,\tilde M)\to 0 \\
 0\to B_{k+1}\to C_{k+1} \to B_{k}\to 0\\
 \hskip 9mm \cdots \ etc \, \cdots \\
 0\to H_{n+1-k}(\tilde W,\tilde M) \to C_{n+1-k} \to B_{n-k}\to 0
\end{array}
$$
As $H_k(\tilde W,\tilde M)$ is finitely generated and projective, so are all the above modules
and one has 
$$
[H_{k}(\tilde W,\tilde M)]= -[B_k]=[B_{k+1}]=-[B_{k+2}]\cdots = (-1)^{r+1}[B_{k+r}]  \, .
$$
As, $[H_{n+1-k}(\tilde W,\tilde M)]=-B_{n-k}$, one has  
\beq{PcobL2-e2}
[H_{n+1-k}(\tilde W,\tilde M)] = (-1)^{n}[H_{k}(\tilde W,\tilde M)] \, . 
\eeq
Now, by \pe duality, 
$$
H_{n+1-k}(\tilde W,\tilde M)=H_{n+1-k}(W,M;\bbz\pi)
\approx H^{k}(W,M';\bbz\pi) \, .
$$
To compute the latter, one uses the handle decomposition $\calh'$ of $(W,M')$ which 
is the dual handle decomposition to $\calh$. 
Then, $\calh'$ has no handle of index $<k$. Therefore
\begin{eqnarray*}
 H^{k}(W,M';\bbz\pi) &\approx & \ker\big(C_k(\tilde\calh')^*\fl{\delta} C_{k+1}(\tilde\calh')^*\big) \nonumber
 \\ &=&
\{\alpha\in C_k(\tilde\calh')^*\mid\alpha\pcirc\partial=0\} \nonumber \\ &\approx & H_k(\tilde W,\tilde M')^* \, .
\end{eqnarray*}
Thus, $ H_k(\tilde W,\tilde M')\approx H^{k}(W,M';\bbz\pi)^*$ and, by \eqref{PcobL2-e1} and \eqref{PcobL2-e2}, one has
\begin{eqnarray*}
w_k(M') &=& w_k(M) + [H_k(\tilde W,\tilde M)] - [H_k(\tilde W,\tilde M')] \\  &=&
w_k(M) + [H_k(\tilde W,\tilde M)] - [H^{k}(W,M';\bbz\pi)]^* \\  &=&
w_k(M) + [H_k(W,M)] - [H_{n+1-k}(\tilde W,\tilde M)]^* \\  &=&
w_k(M) + [H_k(W,M)] + (-1)^{n+1} [H_k(W,M)]^*   \ ,
\end{eqnarray*}
which proves the proposition. 
\end{proof}

The rest of this section is devoted to the proof of the following result.

\begin{Theorem}\label{TcobL2}
Let $M$ be a closed connected \has{k} manifold of dimension $n\geq 2k\geq 6$. 
The following conditions are equivalent.
\begin{itemize}
 \item[(a)] $M$ is cohomology $k$-antisimply cobordant to an \as{k} closed manifold.
 \item[(b)] $[w_k(M)]=0$ in $H^{n+1}(\bbz_2;\tilde K_0(\bbz\pi_1(M),\omega_M)$. 
\end{itemize}
\end{Theorem}

We need some preliminary results before starting the proof.
For two integer $s\leq t$, we denote by $\ecs{t}{s}$ the class of CW-spaces which are \ccs{s}{t}.

\begin{Lemma}\label{Ldececs}
Let $X=X_-\cup X_+$, $X_-\cap X_+=X_0$ be a CW-decomposition. Suppose that $X_\pm$ and $X_0$ are in $\ecs{s}{t}$
and that $H^{s-1}(X_+;R)\to H^{s-1}(X_0;R)$ is onto for any $\bbz\pi_1(X)$-module $R$. Then $X\in\ecs{s}{t}$. 
\end{Lemma}

\begin{proof}
Let $R$ be $\bbz\pi_1(X)$-module and consider the Mayer-Vietoris sequence
$$
\dia{ 
H^{j-1}(X_0;R) \ar[r]^(.55){\delta^{j-1}} & H^j(X;R) \ar[r] & H^j(X_-;R)\oplus H^j(X_+;R)
}
$$
By our hypotheses, $X_0\in\ecs{s}{t}$ and $\delta^{s-1}=0$, which implies the lemma.
\end{proof}

\begin{Lemma}\label{LcobL2}
Let $Y$ be a connected finite cell complex of dimension $k-1\geq 2$ and let 
$\eta$ be a stable vector bundle over $Y$.
Let $n\geq 2k$ be an integer.
Let $\calp\in \tilde K_0(\bbz\pi_(Y),\omega_\eta)$ being $n$-self-dual.
Suppose that $\calp$ represents $0$ in $H^{n+1}(\bbz_2;\tilde K_0(\bbz\pi(Y),\omega_\eta))$.
Then, there exists an $(Y,\eta$)-referred compact ($n+1$)-dimensional manifold $(A,q)$ (see below) such that
\begin{itemize}
 \item[(a)] $M=\bd A$ is \has{k} and satisfies $q_*(w_k(M))=\calp$.
 \item[(b)] $A\in\ecs{k}{n-k}$. 
 \end{itemize}
\end{Lemma}

Here, an $(Y,\eta$)-referred manifold $(A,q)$ means that $q\:A\to Y$ and as well as $q_{|\bd A}$ are both
$(Y,\eta)$-reference maps.

\begin{proof}
We set $(\pi,\omega)=(\pi_1(Y),\omega_\eta)$. All spaces are equipped with a $(Y,\eta)$-reference map, so we use Convention~\ref{Conv} throughout the proof. Reference maps are often dropped from the notation. 
 
We refer to the proof of \proref{PreaPoin} for $r=n$ and use its notations.
There, a \pe pair $(\bbu,\bbp)$ of formal dimension $n+1$
is constructed, together with an $(Y,\eta)$-reference map $q\:\bbu\to Y$, such that
$\wa(\bbu)=\calp$ and $\wa(\bbp) = \calp + (-1)^{n}\calp^*$.
Moreover, $\bbp$ is \has{k} with $w_k(\bbp)=\calp$. The space $\bbu$ is a completion of the one-end open manifold $U$
of \eqref{Trea-e1}, satisfying $\wa(U)=\calp$. We mentioned that $U=U_0\cup E_0$ where $(U_0,q_{|U_0})$ is a 
stable $\eta$-thickening of $Y$. Thus, $\bbu=U_0\cup \bbe$, where $(\bbe,\bd U_0,\bbp)$ is a \pe cobordism with $\wa(\bbe)=\calp$. As the pair $(\bbu,\bbe)$ is highly connected, the restriction $q_{\scriptscriptstyle\bbe}\:\bbe\to Y$  of $q$ to $\bbe$ is an $Y$-reference map. By \corref{CPdecNor2}, $q_{\scriptscriptstyle\bbe}$
is $\eta$-normal.

The hypothesis that $\calp$ represents $0$ in $H^{n+1}(\bbz_2;\tilde K_0(\bbz\pi,\omega))$ means that \\
$\calp=-\calq-(-1)^{n+1}\calq^*$ for some $\calq\in \tilde K_0(\bbz\pi,\omega)$. 
We now use \proref{PreaPoin} for $r=n+1$, replacing $\calp$ by $\calq$. 
One gets a an $(Y,\eta$)-referred \pe complex (without boundary) 
$(\bbp_1,q_1)$ of formal dimension $n+1$, having an $\eta$-normal \pe decomposition
$$
\bbp_1 \simeq N_1 \cup \bbt_1 \ , \ N_1\cap \bbt_1 = \bd N_1  = \bd\bbt_1 \ ,
$$
where $(N_1,{q_1}_{|N_1})$ is a stable $\eta$-thickening of $Y$ in dimension $n+1$ and $\bbt_1$ 
is an $(Y,\eta)$-referred \pe space with boundary $\bd(\bbt_1)=\bd N_1$
such that $\wa(\bbt_1) = \calq + (-1)^{n+1}\calq^*$.

As $q$ and $q_1$ are $\eta$-normal, one can choose Spivak classes $(\bbu)\in\calh_{n+1}(\bbu,\bd\bbu)$ and 
$(\bbp_1)\in\calh_{n+1}(\bbp_1)$. By \proref{PPdecNor} (ii), they induce Spivak classes 
$(U_0)\in\calh_{n+1}(U_0,\bd U_0)$ and $(N_1)\in\calh_{n+1}(N_1,\bd N_1)$. By \proref{LSpiThick},
there exists a simple homotopy equivalence $f\:(U_0,\bd U_0)\to (N_1,\bd N_1)$ commuting with the $Y$-reference
maps and satisfying $\calh_*f((U_0))=(N_1)$. Form the space 
$$
\bba = \bbe \cup_{\bd(U_0)} \calc(f) \cup_{\bd N_1=\bd \bbt_1} \bbt_1 \, ,
$$
where $\calc(f)$ is the mapping cylinder of $f_{|\bd U_0}$:
$$
\calc(f) = \bd U_0 \times [0,1] \dcup \bd N_1 \big/ \{(x,1)\sim f(x)\} \, .
$$
Set $\bbe^+ = \bbe \cup \calc(f)$ and $\bd\bbe^+ = \bbp \dcup \bd\bbt_1$.

That $f$ commutes with the reference maps on $\bbe^+\simeq \bbe$ and $\bbt_1$
implies that these references maps extend to a map 
$q_{\scriptscriptstyle\bba}\:\bba\to Y$. This map $q_{\scriptscriptstyle\bba}$ is an $Y$-reference map.
Indeed, all fundamental groups under consideration are isomorphic and identified with~$\pi$. One has a morphism of 
of Mayer-Vietoris sequences
$$
\dia{ 
H_j(\widetilde{\partial N_1}) \ar[r] \ar[d]^{(q_N)_*}   & 
H_j(\tilde\bbe^+)\oplus H_j(\tilde\bbt_1) \ar[r]  
\ar[d]^{(q_{\scriptscriptstyle\bbe})_*\oplus (q_{\scriptscriptstyle{\bbt_1}})_*}    & 
H_j(\tilde\bba) \ar[d]^{(q_{\scriptscriptstyle\bba})_*}
\ar[r]^(.45
){\partial_{MV}} & H_{j-1}(\widetilde{\partial N_1})  \ar[d]^{(q_N)_*}
\\ 
H_j(\tilde Y) \ar[r] & H_j(\tilde Y)\oplus H_j(\tilde Y) \ar[r] & H_j(\tilde Y) \ar[r]^(.45)0 & H_{j-1}(\tilde Y)
}
$$
where the bottom line is the Mayer-Vietoris sequence of $Y=Y\cup Y$, $Y\cap Y=Y$.
The homomorphism $(q_N)_*\: H_j(\widetilde{\partial N_1})\to H_j(\tilde Y)$ is an isomorphism for $j\leq n+2-k > k$,
whence $\partial_{MV}=0$ on this range. As $q_{\scriptscriptstyle\bbe}$ and $q_{\scriptscriptstyle{\bbt_1}}$
are $Y$-reference maps, one deduces that $q_{\scriptscriptstyle\bba}$ is an $Y$-reference map.

We now refer to Diagram~\eqref{PdecNor-dia} and \proref{PPdecNor}, 
with $X_-=\bbe^+$, $X_+=\bbt_1$ and $Z_0=\bd\bbt_1$. 
Since $\calh_*f((U_0))=(N_1)$, one has $\partial_\calh((\bbe^+),-(\bbt_1))=0$.
By \corref{CPdecNor},  $(\bba,\bbp)$ is a \pe pair of formal dimension~$n+1$ and $q_{\scriptscriptstyle\bba}$ 
is an $(Y,\eta)$-reference map.

We claim that $\bba\in \ecs{k}{n-k}$. Indeed, one has the following
\begin{itemize}
 \item $\bbe^+\in \ecs{k}{n-k}$. For, $E_0\in\ecs{k}{n-k}$ by \proref{Pdecomp} (for $r=n+1$), so
 $\bbe^+\simeq\bbe\simeq E_0\in \ecs{k}{n-k}$.
 \item $\bbt_1\in \ecs{k}{n+1-k}$ by \proref{PreaPoin} (for $r=n+1$).
 \item $\bd U_0 \in \ecs{k}{n-k}$ by \remref{RBdThick}.
\end{itemize}
As the inclusion $\bd U_0\hookrightarrow\bbt_1$ homotopy commutes with the $Y$-reference maps, it induces
a homomorphism $H^{k-1}(\bbt_1;R)\to H^{k-1}(\bd U_0;R)$ which is onto for for any $\bbz\pi_1(X)$-module $R$.
By \lemref{Ldececs}, we deduce that $\bba \in \ecs{k}{n-k}$.

By Siebenmann's sum theorem \cite[Theorem~6.5]{SieTh}, one has
$$
\wa(\bba) = \wa(\bbe) + \wa(\bbt_1) = \calp + \calq + (-1)^{n+1}\calq^* =0 \, .
$$
On the other hand, one has $\wa(\bbp)=\calp+ (-1)^n\calp^*$ by \proref{PreaPoin}. 
Since $\calp=(-1)^{n+1}\calp^*$ by hypothesis, one has $\wa(\bbp)=0$.
By \cite[Theorem~F]{WallFin}, $(\bba,\bbp)$ has the homotopy type of a {\it finite} \pe pair. 
The surgery obstruction $\sigma(\bba,q_{\scriptscriptstyle\bba})$ (see~\ref{MPfnormal}) is thus defined in 
$L^h_{n+1}(\pi,\pi,\omega)=0$ by Wall's ($\pi-\pi$)-theorem.
Thus, up to homotopy equivalence,
one may replace $(\bba,\bbp)$ is an $(Y,\eta)$-referred {\it compact manifold pair} $(A,M)$ of dimension $n+1$.
The manifold pair $(A,M)$ has the required properties for \lemref{LcobL2}.
\end{proof}

\begin{proof}[Proof of \thref{TcobL2}]
That $(b)\!\Rightarrow\! (a)$ is guaranteed by \proref{PcobL2}.
For the converse, let us start with some preliminary constructions. Choose a handle decomposition of $M$ 
and let $N$ be the union of handles of index $\leq k-1$. Then, $N$ is a stable $n$-thickening
of a connected finite subcomplex $Y$ of $M$ of dimension $k-1$. By the classification of stable thickening
(see the proof of \thref{PWallThick}), $N$ is determined up to diffeomorphism
by the the stable vector bundle $\nu$ over $Y$, i.e. the restriction of the stable normal bundle of $M$ over $Y$.
Let $(\pi,\omega)=(\pi_1(Y),\omega_\nu)$. The inclusion $Y\subset M$, which is ($k-1$)-connected, induces an identification
$(\pi_1(M),\omega_M)\approx (\pi,\omega)$. Using this identification, let $\calp=w_k(M)\in\tilde K_0(\bbz\pi,\omega)$.
Set $T=M\psetminus {\rm int} N$. 

Let $(A,q)$ be an $(Y,\nu)$-referred compact manifold of dimension $n+1$, as produced by \lemref{LcobL2}, 
so that $\bd A$ is a closed connected \has{k} manifold with $w_k(\bd A)=-\calp \in\tilde K_0(\bbz\pi,\omega)$
(identification via $q$). Let $M_0=\bd A$. 
[Note the asymmetry of the pairs $(M,Y)$ and $(M_0,Y)$: the pair $(M,Y)$ is ($k-1$)-connected but there is no retraction
of $M$ onto $Y$, while $M_0$ retracts onto $Y$ but the pair $(M_0,Y)$ is only weakly ($k-1$)-connected 
(see \lemref{Lrefretr})]. 
By \proref{Pdecomp}, there is a manifold decomposition
$$
M_0 = N_0 \cup T_0 \quad , \quad N_0\cap T_0 = \bd N_0 = \bd T_0
$$
where $(N_0,q_{|N_0})$ is a stable $\nu$-thickening of $Y$ and $T_0$ is \has{k} with $w_k(T_0)=w_k(\bd A)$.
Using a collar neighborhood of $M_0$ in $A$, the inclusion $N_0\hookrightarrow M_0$ may be extended
to an embedding $\alpha\: N_0\times [0,2]\to A$ with $\alpha^{-1}(\bd A)=N_0\times \{0\}$. 
Let $A_0=A\psetminus (\rm{int} N_0 \times [0,2))$. Note that $A_0$ is homeomorphic to $A$ (diffeomorphic modulo
rounding corners). 

By \proref{PWallThick}, there is a degree-one diffeomorphism $h\:N\to N_0$. Let $h_\bd \:\bd N \to \bd N_0$ 
be restriction of $h$ to $\bd N$. Form the space
$$
W = T\times [0,1] \cup_{\beta} A_0 \, ,
$$
where $\beta\:\bd T\times [0,1] \to \bd T_0\times [0,1]$ is the diffeomorphism given by 
$\beta(x,t)=(\alpha\pcirc h_\bd (x),t)$. 
After rounding the corners (see appendix of \cite{BSD}), $W$ is a smooth cobordism between the manifolds
$$
\calm = T \cup_{h_\bd} T_0 \ \hbox{ and } \  \calm_1 \approx T \cup_{h_\bd} N_0  
$$
(since $\bd N_0\times [1,2] \cup N_0\times\{2\}\approx N_0$).
\thref{TcobL2} will follow from the three assertions
\begin{itemize}
 \item[(A)] $\calm$ is \as{k}.
 \item[(B)] $\calm_1$ is diffeomorphic to $M$.
 \item[(C)] $W$ is a \has{k} cobordism.
\end{itemize}

We now start proving the above assertions. There is an embedding $Y\hookrightarrow \bd N$.
We leave to the reader to prove, using Van Kampen's theorem,
that the various fundamental groups $\pi_1(W)$, $\pi_1(\calm)$, are thus all identified with $\pi_1(Y)$.

\sk{1}\noindent 
{\em Proof of (A).} \ 
We first prove that $\calm$ is \has{k} (i.e. $\calm\in\ecs{k}{n-k}$).
Let $R$ be a $\bbz\pi$-module. Given that $T_0\in\ecs{k}{n-k}$ and the exact sequence
$$
H^{j-1}(T_0;R) \fl{\delta}   H^j(\calm,T_0;R) \to H^j(\calm;R) \to H^j(T_0;R) \ ,
$$
it is enough to prove that $\delta$ is surjective when $k\leq j\leq n-k$. By excision, one has
$$
H^j(\calm,T_0;R) \approx H^j(T,\bd N;R) \approx  H^j(M,N;R) \, .
$$
As $M\in\ecs{k}{n-k}$ and $\dim Y = k-1$, one deduces from the exact sequence
$$
H^{j-1}(N;R) \to H^j(M,N;R) \to H^j(M;R)
$$
that $H^j(\calm,T_0;R)\approx H^{j-1}(N;R)=0$ for $j>k$. For $j=k$, one uses the following commutative diagram
\beq{TcobL2-L1-dia}
\dia{ 
H^{k-1}(T_0;R) \ar[d]^(0.43)\delta \ar[r]^(0.47){a} & H^{k-1}(\bd T;R) \ar[d]^(0.43){\delta_1} \ar@{<-}[r] & 
H^{k-1}(N;R) \ar[d]^(0.43){\delta_2}
\\ 
H^k(\calm,T_0;R) \ar[r]^(0.47)\approx & H^k(T,\bd T;R) \ar@{<-}[r]^(0.54)\approx & H^k(M,N;R)
}
\eeq
where the horizontal arrows are induced by inclusions, via the identifications \\ 
$\bd T=\bd N \approx \bd N_0 = \bd T_0$. 
The homomorphism $\delta_2$ is onto since $H^k(M;R)=0$. This implies that $\delta_1$ is onto. Now, the homomorphism $a$ may be identified with
the homomorphism $a_0$ in the commutative diagram
$$
\dia{ 
 H^{k-1}(M_0;R)  \ar[rr]^{(\sigma_0)^*} \ar[d] &&  H^{k-1}(N_0;R) \ar[d]^(0.45)b  
\\ 
H^{k-1}(T_0;R) \ar[r]^(0.47){a_0}  &  H^{k-1}(\bd T_0;R) \ar[r]^\approx & H^{k-1}(\bd N_0;R)
} .
$$
The homomorphism $(\sigma_0)^*$ is onto, since the inclusion $\sigma_0\:N_0\to M_0$ admits a homotopy retraction 
(see \lemref{Lrefretr}). The homormorphism $b$ induced by the inclusion is also onto for, by \pe duality
$$
H^k(N_0,\bd N_0;R) \fl{\approx} H_{n-k}(N_0;R) = 0  \, ,
$$
since $n-k>k-1$. Therefore, $a$ is onto in Diagram \eqref{TcobL2-L1-dia}, implying that $\delta$ is onto, 
which proves that $\calm$ is \has{k}.

We now compute the antisimple obstruction $w_k(\calm)$. 
By excision, one has isomorphisms
\beq{DexcisMM0}
\dia{
H_*(\tilde \calm,\widetilde{\bd N}) \ar@{<-}[r]^(.39)\approx & H_*(\tilde T,\widetilde{\bd N}) \oplus H_*(\tilde T_0,\widetilde{\bd N}) \ar[r]^\approx & H_*(\tilde M,\tilde N) \oplus H_*(\tilde M_0,\tilde N_0)
}
\eeq
By construction, the pair $(M,N)$ is ($k-1$)-connected and so is the pair $(M_0,N_0)$ by \lemref{Lrefretr}.
Therefore, \eqref{DexcisMM0} implies that $(\calm,\bd N)$ is ($k-1$)-connected. The first isomorphism 
of~\eqref{DexcisMM0} for $*=k$ then implies that $\pi_k(\calm,N) \approx \pi_k(M,N)\oplus \pi_k(M_0,N_0)$, which,
by Definition~\eqref{defwk}, implies that 
$$
w_k(\calm)=w_k(M)+w_k(M_0) = \calp - \calp =0 \, .
$$
As $n\geq 6$, Assertion~(A) follows from \proref{Pcasas}.

\sk{1}\noindent 
{\em Proof of (B).} \
Let $c\:\bd N\times [0,1)\to N$ and $c_0\:\bd N_0\times [0,1)\to N_0$ be embeddings of open collar neighborhoods
of $\bd N$ and $\bd N_0$ in $N$ and $N_0$. Let $T^+=T\cup_{\bd T=\bd N} \bd N\times [0,1)$. As smooth manifolds,
$M$ and $\calm_1$ may be presented as follows
$$
M = T^+ \dcup N\big/ \{(x,t)\sim c(x,t) \hbox{ for }  (x,t)\in \bd N\times (0,1)\}
$$
and
$$
\calm_1 = T^+ \dcup N_0\big/ \{(x,t)\sim c_0(h_\bd(x),t) \hbox{ for }  (x,t)\in \bd N_0\times (0,1)\}  \, .
$$
Thus, a diffeomorphism $s\: M\fl{\approx} \calm_1$ may be defined as
$$
s(x) = 
\left\{
\begin{array}{ll}
x & \hbox{if $x\in T^+$} \\
h(x) & \hbox{if $x\in {\rm int\,}N$.}
\end{array} 
\right.
$$


\sk{1}\noindent 
{\em Proof of (C).} \
We consider the morphism of Mayer-Vietoris sequences for the inclusion of $4$-ads
$$
(\tilde \calm,\tilde T,\tilde T_0,\widetilde{\bd N}) \hookrightarrow
(\tilde W,\tilde T\times [0,1],\tilde A_0,\widetilde{\bd N}\times [0,1]) \, .
$$
For $j\leq k-1$, all the homomorphisms, except possibly $H_j(\calm)\to H_j(W)$, are isomorphisms .
But then the latter is an isomorphism by the five lemma. The same happens for the inclusion
$$
(\tilde \calm_1,\tilde T,N,\widetilde{\bd N}) \hookrightarrow
(\tilde W,\tilde T\times [0,1],\tilde A_0,\widetilde{\bd N}\times [0,1]) \, .
$$
Therefore, both pairs $(W,\calm)$ and $(W,\calm_1)$ are weakly ($k-1$)-connected. 
The proof that $W\in\ecs{k}{n-k}$ is the same as that of $\calm\in\ecs{k}{n-k}$ (see the proof of (A)): just replace 
$\calm$ by $W$ and $T_0$ by $A_0$. 
\end{proof}

\subsection{Examples}\label{Sexpl}

Let  $Y$ be a connected finite cell complex of dimension $k-1\geq 2$ and let $\eta$ be a stable vector bundle over $Y$.
Let $(\pi,\omega)=(\pi_1(Y),\omega_\eta)$. 
Set $H^+=H^{2j}=H^{2j}(\bbz_2,\tilde K_0(\pi,\omega))$ and 
$H^-=H^{2j+1}=H^{2j+1}(\bbz_2,\tilde K_0(\pi,\omega))$.

Let $\calp$ be a  $\pm$-symmetric element of $\tilde K_0(\bbz\pi,\omega)$,
i.e. $\calp$ satisfies $\calp=\pm\calp^*$. Suppose that 
$\calp$ represents zero in $H^\pm$. By \thref{Trea}, 
there exist $(Y,\eta)$-referred closed connected manifolds $(M,g)$ of dimension $r\geq 2k$
such that $M$ is \has{k} with $g_*(w_k(M))=\calp$. 
By \thref{TcobL2}, when $r\geq 6$,
such manifolds are cohomology $k$-antisimply cobordant to an \as{k} closed manifolds.

A source of such examples is given by the odd torsion $T_{\rm odd}$ of $\tilde K_0(\bbz\pi)$.
Indeed, $T_{\rm odd}$ splits under the involution $*$ into a direct sum of
eigenspaces $T_{\rm odd}^{\pm}$ for the eigenvalues $\pm 1$. This produces $\pm$-symmetric element of $\tilde K_0(\bbz\pi,\omega)$ which represent zero in $H^\pm$,
since the latter is a $2$-group. The group  $T_{\rm odd}$ is not zero for many finite groups,
such as $p$-groups for $p$ an odd prime, or the symmetric groups (see \cite{OliverProj} for a survey and references).

More essential examples may arise when $\pi$ is a finite $2$-group. For example, when $\pi=(\bbz_2)^3$, then
$\tilde K_0(\bbz\pi)\approx\bbz_2$ \cite[Theorem~12.9]{WallNorms}. Thus, $H^\pm\approx\bbz_2$. On the other hand,
the homomorphism
\beq{deltan}
\delta^{(n)} = \delta_R^{(n)} : H^{n+1} \to L_n^h(\pi,\omega)
\eeq
of \eqref{ranicki} vanishes for $n\equiv 2\, ({\rm mod} 4)$ \cite[Theorem~C]{HamMil} (warning: 
the conventions for $H^\pm$ in \cite{HamMil} are the opposite of ours, since the authors use there the involution
$\calp\mapsto -\calp^*$). By \thref{Trea} and \proref{PcobL2}, for $n=4j+2\geq 2k$, there exists
an $(Y,\eta)$-referred closed connected $n$-manifolds $(M,g)$ which is \has{k} but not
cohomology $k$-antisimply cobordant to an \as{k} manifold. As \cite[Theorem~C]{HamMil} applies to a large class of
finite $2$-groups, one might be able to get other examples.

When $\pi$ is the generalized quaternion group $Q2^s$ ($s\geq 3$), one sees in \cite[\S~6 and~(1.5)]{HamMil}
that $H^\pm\approx\bbz_2$. On the other hand, 
the homomorphism $\delta^{(n)}$ of~\eqref{deltan} vanishes for 
$\not\equiv 1\, ({\rm mod} 4)$ \cite[Theorem~D]{HamMil}. As above, one deduces from \thref{Trea} and \proref{PcobL2} that, for $2k\leq n\neq 4j+1$, there an $(Y,\eta)$-referred closed connected $n$-manifolds $(M,g)$ which is \has{k} but  not cohomology $k$-antisimply cobordant to an \as{k} manifold.

Finally, \proref{Pcasas}, \thref{TcobL2} together with \proref{PproductFor} implies the following

\begin{Proposition}\label{CCproductFor}
Let $M$ be a closed connected $n$-dimensional manifold ($n\geq 6$) which is \has{k} for $k\geq 3$. 
Let $A$ be a closed connected manifold of dimension $a$ with $a\leq n-2k$. Then the manifold $M\times A$ is \has{(k+a)} and
\begin{itemize}
\item[(1)] if $\chi(A)=0$, then $M\times A$ is \as{(k+a)}
\item[(2)] if $\chi(A)$ is even, then $M\times A$ is cohomology $(k+a)$-antisimply cobordant to an \as{(k+a)} closed manifold. \mancqfd
\end{itemize}
\end{Proposition}

\section{Referred antisimple manifolds}\label{SrefASManif}

Through this section $Y$ will be a fixed connected finite cell complex of dimension $k-1\geq 2$ equipped
with a stable vector bundle $\eta$ over it. We set $(\pi,\omega)=(\pi_1(Y),\omega_\eta)$.
We denote by $\wh(\pi,\omega)$ the Whitehead group $\wh(\pi)$ endowed with the involution $\alpha\mapsto\alpha^*$
using the orientation character $\omega$. 
When we consider an $(Y,\eta$)-referred manifold
$(X,g)$, its fundamental group and orientation character $(\pi_1(X),\omega_X)$ is thus is identified via $g_*$ with 
$(\pi,\omega)$. This gives an identification $\wh(\pi_1(X),\omega_X)\approx\wh(\pi,\omega)$.
Equalities between elements of $\wh(\pi_1(X'),\omega_{X'})$
and $\wh(\pi_1(X''),\omega_{X''})$ for ($Y,\eta$)-manifolds $(X',g')$ and $(X'',g'')$
should be understood to hold in $\wh(\pi,\omega)$ via the identifications (like in Convention~\ref{Conv}).
The intervalle $[0,1]$ is denoted by $I$.

As done in the proof of \proref{Pdecomp}, if $(M,g)$ be an $(Y,\eta$)-referred closed manifold of dimension $r\geq 2k\geq 6$, a homotopy section $\gamma\:Y\to M$ of $g$ gives rise to 
a manifold decomposition $N\cup T\fl{\approx} M$ such that $(N,g_{|N})$ is a stable $\eta$-thickening of $Y$
When $M$ is \as{k}, we get a richer decomposition.

\begin{Proposition}\label{PdecAsm}
Let $(M,g)$ be an $(Y,\eta$)-referred \as{k} closed manifold of dimension $r\geq 2k\geq 6$. 
Then,

\begin{itemize}
 \item[(1)] a homotopy section $\gamma\:Y\to M$ of $g$ gives rise to 
a manifold decomposition
$$
M=N \cup V \cup N^* \ , \  N \cap N^* = \emptyset \ ,
\ V\cap N = \bd N \ , \ V\cap N^* = \bd N^*
$$
such that
\begin{itemize}
 \item[(i)] $(N,g_{|N})$ and $(N^*,g_{|N^*})$  are stable $\eta$-thickenings of $Y$ of dimension~$r$.
 \item[(ii)] the compact $r$-manifold $V$ is an h-cobordism between $\bd N$ and $\bd N^*$. 
\end{itemize}
\item[(2)] The Whitehead torsion $\tau(V,\bd N)\in Wh(\pi,\omega)$ is $r$-self-dual, i.e. satisfies \\
$\tau(V,\bd N)=(-1)^{r+1}\tau(V,\bd N^*)^*$.
\item[(3)]
The class $[\tau(V,\bd N)]\in H^{r+1}(\bbz_2;\wh(\pi_1(Y))$ does not depend on the section 
$\gamma$. 
\end{itemize}  
\end{Proposition}

Thanks to (3) above, the class of $[\tau(V,\bd N)] \in H^{r+1}(\bbz_2;\wh(\pi_1(Y))$ will be denoted by 
$\tau(M,g)$.

\begin{proof}
As in the proof of \proref{Pdecomp}, 
one may suppose that $\gamma\:Y\to M$ is an embedding. Let 
$N$ be a smooth regular neighborhood of $\gamma(Y)$. We thus get the manifold decomposition $M= N\cup T$
of \proref{Pdecomp}.
By general position, $\gamma$ is isotopic to an embedding
$\gamma^*\:Y\to T$. Let $N^*$ be a smooth regular neighborhood of $\gamma^*(Y)$ in ${\rm int}(T)$.
Let $V=M \psetminus {\rm int}(N\cup N^*)$. 

To prove that the cobordism $(V,\bd N,\bd N^*)$ is an h-cobordism, let us 
consider a handle decomposition $\calh$ for $M$
$$
\calh : \ D^n = \calh_0 \subset \calh_1 \subset \cdots \subset \calh_n = M \ ,
$$
where $\calh_j$ is the union of handles of index $\leq j$ in $\calh$. We suppose that $\calh$ is \as{k},
i.e. $\calh_{k-1}=\calh_{n-k+1}$ and that $\calh$ has only one $r$-dimensional handle. 
Consider the dual handle decomposition $\calh^*$ of $\calh$ (see e.g. \cite[p.~394]{MilnorWT}).

By general position, one may assume that $\gamma(Y)\subset N\subset  {\rm int\,}\calh_{k-1}$. 
Let $X=\calh_{k-1}\psetminus {\rm int\,}N$.
The pairs $(M,N)$ and $(M,\calh_{k-1})$ are both weakly ($k-1$)-connected. Therefore, the pair $(\calh_{k-1},N)$
is weakly ($k-1$)-connected. Since the homotopy dimensions of $Y$ and of $\calh_{k-1}$ are both $\leq k-1$,
one deduces that $\gamma\:Y\to \calh_{k-1}$ is a homotopy equivalence. It follows easily that $(X,\bd N,\bd\calh_{k-1})$
is an h-cobordism (see e.g. \cite[proof of Proposition~4.6]{HauJah}).

Similarely, one may suppose that $N^*\subset {\rm int\,}\calh_{k-1}^*$ and get an h-cobordism \\
$(X^*,\bd N^*,\bd\calh_{k-1}^*)$. One has that $V\cap V^*=\bd\calh_{k-1}=\bd\calh_{k-1}^*$, 
so the cobordism $V= X\cup X^*$ is an h-cobordism.

To prove (2), recall from \remref{RBdThick} that $N\approx N_-\times I$
and thus $\bd N \approx \bar N_-\cup N_-\times\{1\}$, where $\bar N_-=N_-\times\{0\} \cup \bd N_-\times I$.
Analogously,  $N^*\approx N_-^*\times I$ 
and $\bd N^* \approx \bar N_-^*\cup N_-^*\times\{1\}$, where $\bar N_-^*=N_-^*\times\{0\} \cup \bd N_-^*\times I$.
For $t=0,1$, consider the embedding $\gamma_t\:Y\to N_-\times\{t\}$ homotopic in $N$ to $\gamma$. 
We do the same for $\gamma_t^*\:Y\to N_-^*\times\{t\}$ homotopic in $N^*$ to $\gamma^*$. 

The two embeddings $\gamma_0$ and  $\gamma_0^*$ are isotopic in $M$. By general position, they are isotopic in $V$. Therefore, $V\approx (\bar N_-\times I)\cup Z$, where $(Z,N_-\times 1,N_-^*\times 1)$ is an h-cobordism.
Set $\bar\gamma_1\:Y\to Z$ to be $\gamma_1$ post-composed with the inclusion $N_-\times\{1\}\hookrightarrow Z$
and $\bar\gamma_1^*\:Y\to Z$ to be $\gamma_1^*$ post-composed with the inclusion $N_-^*\times\{1\}\hookrightarrow Z$
One has
$$
\tau(V,\bd N) = \tau(Z,N_-\times\{1\}) = \tau(\bar\gamma_1) 
$$
and 
$$
(-1)^{r+1}\tau(V,\bd N) = \tau(V,\bd N^*) = \tau(Z,N_-^*\times\{1\}) = \tau(\bar\gamma_1^*) \, . 
$$
Since $\gamma_1$ and $\gamma_1^*$ are homotopic in $\bd Z$, this proves (2). 

To prove (3), let $\delta\:Y\to M$ be another homotopy section of $g$, giving rise to a manifold decomposition
$M\approx N(\delta) \cup V(\delta)\cup N^*(\delta)$ which we have to compare with the decomposition
$M\approx N(\gamma) \cup V(\gamma)\cup N^*(\gamma)$. As above, by general position and given the connectivity of 
$(M,N(\gamma))$ and $(M,N(\delta))$, there is a self-homotopy equivalence $\phi$ of $Y$ such that 
$\gamma=\delta\pcirc\phi$. Therefore, we shall have $N(\delta)\subset {\rm int} N(\gamma)$
and $N^*(\delta)\subset {\rm int} N^*(\gamma)$, 
whence h-cobordisms $(S,\bd N(\delta),\bd N(\gamma))$ and $(S^*,\bd N^*(\delta),\bd N^*(\gamma))$. One has
$V(\delta) = S \cup V(\gamma) \cup S^*$. Therefore
\begin{eqnarray*}
 \tau(V(\delta),\bd N(\delta)) &=& \tau(V(\gamma),\bd N(\gamma)) + \tau(S,\bd N(\delta)) + \tau (S^*,\bd N^*(\gamma))
 \\ &=&
 \tau(V(\gamma),\bd N(\gamma)) + \tau(S,\bd N(\delta)) + (-1)^{r+1}\tau(S,\bd N(\delta))^* \, ,
\end{eqnarray*}
which proves (3).
\end{proof}

\begin{Remark}\label{RTdoub}\rm
In case $\tau(M,g)=0$, then, by the s-cobordism theorem, $V\cup N^*$ is diffeomorphic to $N$ and $M$ is diffeomorphic to
the gluing of two copies of $N$ by a self-diffeomorphism of $\bd N$. We call such an antisimple manifold a 
\dfn{twisted double}.\mancqfd
\end{Remark}

Two $(Y,\eta)$-referred manifold $(M_i,g_i)$ ($i=0,1$) are called \dfn{h-cobordant} if there is an h-cobordism
$(W,M_0,M_1)$ admitting a $k$-connected $\eta$-normal map $G\:W\to Y$ extending $g_0$ and $g_1$. 

\begin{Proposition} 
Let $(M,g)$ be an $(Y,\eta$)-referred \as{k} closed manifold of dimension $r\geq 2k\geq 6$. 
Then,
\begin{itemize}
 \item[(1)] The class $\tau(M,g)\in H^{r+1}(\bbz_2;Wh(\pi_1(Y))$ is an invariant of the h-cobordism class of $(M,g)$.
\item[(2)] $\tau(M,g)=0$ if and only if $(M,g)$ is h-cobordant to a twisted double.
\end{itemize} 
\end{Proposition}

\begin{proof}
Let $(W,G)$ be an ($Y,\eta$)-referred h-cobordism between $(M_0,g_0)$ and $(M_1,g_1)$. 
Let $\Gamma\:Y\to W$ be a homotopy section of $G$. Composed with the projections on $M_i$, it gives rise to homotopy sections $\gamma_i\:Y\to M_i$ of $g_i$. One can also construct a homotopy $\hat\Gamma\:Y\times I \to W$  between 
$\gamma_0$ and $\gamma_1$ and a homotopy $\hat\Gamma^*\:Y\times I \to W$  between 
$\gamma_0^*$ and $\gamma_1^*$. By general position, this can be done by embeddings and we thus get  
a decomposition
$$
W \approx (N\times I) \cup Z \cup (N^*\times I) 
$$
where $(Z,V(\gamma_0),V(\gamma_1))$ is an h-cobordism.  One has

\noindent\hskip -6mm
\begin{minipage}{12.7cm}
{\small
\begin{eqnarray}
 \tau(V(\gamma_0),\bd N(\gamma_0)) + \tau(Z,V(\gamma_0))  &=& 
\tau(V(\gamma_1),\bd N(\gamma_1)) + \tau(Z,V(\gamma_1)) \nonumber  \\ &=&
\tau(V(\gamma_1),\bd N(\gamma_1)) + (-1)^{r}\tau(Z,V(\gamma_0))^* \label{eqhcob}
\end{eqnarray}
}
\end{minipage}\\
which proves (1).

If $(M,g)$ is h-cobordant to a twisted double, then $\tau(M,g)=0$ by (1) already proven.
Conversely, suppose that $\tau(M,g)=0$. This means that there is a manifold decomposition
$M\approx N \cup V \cup N^*$ with $(V,\bd N,\bd N^*)$ being an h-cobordism with 
$\tau(V,\bd N)= \alpha+ (-1)^{r+1}\alpha^*$ for some $\alpha\in\wh(\pi,\omega)$. Let $(Z,M,M')$ be an h-cobordism
with $\tau(Z,M)=-\alpha$. Using \eqref{eqhcob}, this proves that $\tau(V',N')=0$ and thus $M'$ is a twisted double.
\end{proof}

We now turn our attention to the construction of antisimple manifolds which are not h-cobordant to
twisted doubles. Analogously to Ranicki's exact sequence \eqref{ranicki}, the main ingredient is the {\it Rothenberg exact sequence} 
\cite[Proposition~4.1]{Sha}
\beq{Rothenberg}
L_{r+1}^s(\pi,\omega) \to L_{r+1}^h(\pi,\omega) \fl{T}  H^{r+1}(\bbz_2;Wh(\pi,\omega)) \fl{\delta_{\rm Rot}} L_{r}^s(\pi,\omega)
\eeq

\begin{Proposition}\label{PAsreali}
Let $a\in H^{r+1}(\bbz_2;Wh(\pi,\omega))$ such that $\delta_{\rm Rot}(a)=0$. Then, for $r\geq 2k\geq 6$ there exists an $(Y,\eta)$-referred \as{k} closed manifold $(M,g)$ with $\tau(M,g)=a$.
\end{Proposition}

\begin{proof}
The map $T$ in \eqref{Rothenberg} may be geometrically described as follows (see \cite[p.~313]{Sha}). Let $(N,g)$ be a stable $\eta$-thickening
of $Y$ of dimension $r$. Set $\bd (N\times I)\approx N_- \cup N_+$, where  $N_-=N\times\{0\}\cup \bd N\times [0,1/2]$
and $N_+=N\times\{1\}\cup \bd N\times [1/2,1]$.
An element $\sigma\in  L_{r+1}^h(\pi,\omega)$ may be represented as the surgery obstruction of a degree one normal map $f\: (X,X_-,X_+)\to (N\times I,N_-,N_+)$, where $X$ is a compact ($r+1$)-dimensional manifold with $\bd X = X_-\cup X_+$,
such that $f_-=f_{|X_-}\:X_-\to N_-$ is a diffeomorphism and $f_+=f_{|X_+}\:X_+\to N_+$ is a homotopy equivalence.
One has $T(\sigma)=\tau(f_+)$ (via the identification by $g_*$).

As mentioned in \remref{RBdThick}, $N=N^{(n)}$ is itself an s-cobordism between an $\eta$-thickening 
$N^{(n-1)}$ of $Y$ and itself. It follows that $X_+$ is an h-cobordism between two copies 
$N^{(n-1)}_\pm$ of $N^{(n-1)}$, with $\tau(N_+,N^{(n-1)})=-\tau(f)$. One deduces that $M=\bd X$
is an \as{k} closed manifold, ($Y,\eta)$-referred via $g$, with $\tau(M,g)=a$ (signs are irrelevant
since $H^{r+1}(\bbz_2;Wh(\pi,\omega))$ is a group of exponent $2$).
\end{proof}

\begin{Examples}\rm
(1)
$L_{\rm odd}^h(\pi)=0$ if $\pi$ is a finite group of odd order \cite{BakOdd} ($\omega$ is then trivial). Hence, by \proref{PAsreali}, any class $a\in H^{2j}(\bbz_2;Wh(\pi))$ may be realized as $\tau(M,g)$ for an $(Y,\eta)$-referred
closed \as{k} manifold of dimension $n=2j\geq 6$. 
It follows from \cite[\S~4]{Bass1} that  $H^{2j}(\bbz_2;Wh(\pi))\neq 0$ if $\pi$ is finite abelian, 
of exponent $\neq 1,2,3,4$ or $6$.

(2) Let $Y_1,\dots,Y_s$ be finite CW-complexes of dimension $\leq k-1$, such that $\pi_1(Y_i)=G_i$
is a finite group of odd order. Let $Y=Y_1\vee\cdots Y_s$ and $G=\pi_1(Y)\approx G_1 * \cdots * G_s$.
One has $L_{\rm odd}^h(G)=0$ by \cite[Theorem~5]{CappellMV}. Hence, as in (1) above, 
any class $a\in H^{2j}(\bbz_2;Wh(G))$ may be realized as $\tau(M,g)$ for an $(Y,\eta)$-referred
closed \as{k} manifold of dimension $n=2j\geq 6$.
Recall that $H^{2j}(\bbz_2;Wh(G))\approx\bigoplus_{i=1}^s H^{2j}(\bbz_2;Wh(G_i))$ by \cite{StallingsWhFP}.

(3)  When $Y$ is acyclic, an $(Y,\eta)$-referred closed \as{k} manifold is a homology sphere. As an example,
consider the \pe $3$-sphere minus a disk, which collapses to an acyclic $2$-dimensional finite complex $Y$ with
$$
\pi_1(Y) = \Delta = \langle a,b\, |\, a^5=b^3=(ab)^2 \rangle  ,
$$
the binary icosahedral group. Let $C_5= \langle t\, |\, t^5=1 \rangle $ 
be the cyclic group of order $5$ with a given generator $t$. There is a homomorphism $\nu\: C_5\to\Delta$
with $\nu(t)=a^2$. One can prove that the induced homomorphism
$$
\bbz_2 \approx H^{2j}(\bbz_2;\wh(C_5)) \fl{\nu_{2j}} H^{2j}(\bbz_2;\wh(\Delta))
$$
is injective \cite[Chapter~5]{HauTh}. By naturality of the Rothenberg exact sequence and Example~(1) above,
this produces \as{3} homology spheres in dimension $r=2j\geq 6$ which are not h-cobordant to a twisted double.
Using (1) and (2) above, we can get such homology spheres with fundamental group a free product of finitely many copies of $\Delta$.

(4) Suppose that $Y$ is the lens space $L(p,q)$ with $p\equiv 3\ {\rm mod\,}4$, or $(p,q)=(5,2)$
and take for $\eta$ the trivial bundle.
Then, any ($Y,\eta$)-referred \as{k} closed manifold $M$
of dimension $r=2j+1\geq 2k\geq 6$ is a twisted double. Indeed, in the decomposition of \proref{PdecAsm},
the h-cobordism $(V,\bd N,\bd N^*)$ is inertial, i.e. $\bd N^*$ is diffeomorphic to $\bd N$. But, there is 
no non-trivial inertial h-cobordism starting from $\bd N = Y\times S^{2j}$  under our hypotheses
\cite[Theorem~6.1 and Lemma~6.5]{Hau2}. Hence $M$ is a twisted double by \remref{RTdoub}.
\end{Examples}


\sk{5}\noindent
\begin{minipage}[t]{6cm}
Jean-Claude HAUSMANN\\
Department of Mathematics\\
University of Geneva\\
CH-1211 Geneva 4,
Switzerland\\
Jean-Claude.Hausmann{@}unige.ch
\end{minipage}

\end{document}